\documentclass{amsart}
\usepackage{latexsym,amsxtra,amscd,ifthen,amsmath,color, multicol,hyperref,caption}
\usepackage{amsfonts}
\usepackage{verbatim}
\usepackage{amsmath}
\usepackage{amsthm}
\usepackage{amssymb}
\usepackage[notcite,notref,final]{showkeys}
 \usepackage[all,cmtip]{xy}
 \usepackage[normalem]{ulem} 
\usepackage{enumerate}
 \usepackage{tikz}
\usetikzlibrary{matrix,arrows}

\numberwithin{equation}{section}

\theoremstyle{plain}
\newtheorem{theorem}{Theorem}[section]
\newtheorem{lemma}[theorem]{Lemma}

\newtheorem{proposition}[theorem]{Proposition}

\newtheorem{corollary}[theorem]{Corollary}

\theoremstyle{definition}
\newtheorem{definition}[theorem]{Definition}
\newtheorem{example}[theorem]{Example}

\newtheorem{notation}[theorem]{Notation}

\newtheorem{remark}[theorem]{Remark}

\makeatletter              
\let\c@equation\c@theorem  
\makeatother

\DeclareMathOperator{\EEnd}{{\underline{\sf End}}}
\DeclareMathOperator{\Fun}{{\sf Fun}}
\DeclareMathOperator{\End}{{\sf End}}
\DeclareMathOperator{\Hom}{{\sf Hom}}
\DeclareMathOperator{\Rep}{{\sf Rep}}
\DeclareMathOperator{\Mod}{{\sf Mod}}
\DeclareMathOperator{\sVec}{{\sf Vec}}
\DeclareMathOperator{\bVec}{\mathbf{Vec}}
\DeclareMathOperator{\Bimod}{{\sf Bimod}}

\DeclareMathOperator{\id}{{\sf id}}

\DeclareMathOperator{\FPdim}{FPdim}

\newcommand{\kk}{\Bbbk}
\newcommand{\red}{\textcolor{black}}

\newcommand{\C}{\ensuremath{\mathcal{C}}}
\newcommand{\D}{\ensuremath{\mathcal{D}}}
\newcommand{\M}{\ensuremath{\mathcal{M}}}

\newcommand{\xto}[1]{\xrightarrow{#1}}

\newcommand{\brk}{\smallbreak}

\begin{document}

\title[Tensor algebras in finite tensor categories]
{Tensor algebras in finite tensor categories}

\author{Pavel Etingof}
\address{Department of Mathematics, Massachusetts Institute of Technology, Cambridge, MA 02139, USA}
\email{etingof@math.mit.edu}

\author{Ryan Kinser}
\address{Department of Mathematics, University of Iowa, Iowa City, Iowa 52242, USA}
\email{ryan-kinser@uiowa.edu}

\author{Chelsea Walton}
\address{Department of Mathematics, The University of Illinois at Urbana-Champaign, Urbana, IL, USA}
\email{notlaw@illinois.edu}

\bibliographystyle{alpha}

\begin{abstract}
This paper introduces methods for classifying actions of finite-dimensional Hopf algebras on path algebras of quivers, and more generally on tensor algebras $T_B(V)$ where $B$ is semisimple.  We work within the broader framework of finite (multi-)tensor categories $\mathcal{C}$, classifying tensor algebras in $\mathcal{C}$ in terms of $\mathcal{C}$-module categories. We obtain two classification results for actions of semisimple Hopf algebras: the first for actions which preserve the ascending filtration on tensor algebras, and the second for actions which preserve the descending filtration on completed tensor algebras. Extending to more general fusion categories, we illustrate our classification result for tensor algebras in the pointed fusion categories ${\sf Vec}_{G}^{\omega}$ and in group-theoretical fusion categories, especially for the representation category of the Kac-Paljutkin Hopf algebra.
 \end{abstract}



\maketitle

\setcounter{tocdepth}{3}

\tableofcontents


\section{Introduction}

Let $\kk$ be an algebraically closed field of characteristic 0. One motivation of this work is to continue the last two authors' study of finite quantum symmetries of path algebras of quivers $\kk Q$. As finite groups are viewed classically as collections of finite symmetries (i.e., automorphisms of finite order) of a given algebra, finite-dimensional Hopf algebras are widely accepted to be an algebraic structure that captures an algebra's {\it finite quantum symmetries}. The main two classes of finite-dimensional Hopf algebras over $\kk$  are those that are {\it semisimple} (as a $\kk$-algebra, that is, all of its modules can be decomposed into a direct sum of simple modules), and those that are {\it pointed} (as a $\kk$-coalgebra, that is, all of its simple comodules are 1-dimensional). The actions of finite-dimensional pointed Hopf algebras on path algebras were investigated previously in  \cite{KinserWalton}, and one aim of the work here is to study actions of semisimple  Hopf algebras on $\kk Q$.  
To achieve this, we establish a broader framework: we analyze tensor algebras in finite multi-tensor categories, which includes actions of finite-dimensional Hopf algebras on path algebras of quivers as a special case.

\brk We begin by providing preliminary results on tensor algebras $T_S(E)$ in finite multi-tensor categories $\C$; here, $S$ is an exact algebra in $\mathcal{C}$ and $E$ is an $S$-bimodule in $\mathcal{C}$; these are referred to as {\it $\C$-tensor algebras} [Definition~\ref{S,E}]. (We can  allow $T_S(E)$ to be in the ind-completion of $\C$, but omit further mention of this technicality throughout the introduction.)  Our first result is that any tensor algebra in $\C$ can be decomposed into {\it minimal} ones in the sense that $E$ is indecomposable [Proposition~\ref{prop:decomp}]. Then one of our main results, Theorem~\ref{thm:param}, classifies (minimal) tensor algebras in a given finite multi-tensor category $\C$; this classification is given in terms of $\mathcal{C}$-module categories $\mathcal{M}$ and (indecomposable) objects in $\mathsf{Fun}_\C(\M,\M)$. The classification in Theorem~\ref{thm:param} is up to {\it equivalence of $\C$-tensor algebras} $T_S(E)$ [Definition~\ref{def:equiv}], which is a notion of equivalence induced by Morita equivalence of the base algebra $S$ and resulting conjugacy class of $E$. This framework and  result are established in Sections~\ref{sec:Cmulti} and~\ref{sec:minimal}.

\smallskip

In the case when $\C$ is multi-fusion,  we  show that there are only finitely many equivalence classes of minimal faithful $\C$-tensor algebras $T_S(E)$, up to Morita equivalence  of $S$ and up to conjugacy class of $E$ [Corollary~\ref{cor:finitelymany}].  
We also study in Section~\ref{sec:CHopf} certain filtration preserving actions of semisimple Hopf algebras on tensor $\kk$-algebras and their completions that do not fit within the framework above, a priori, but can be classified up to equivalence by Theorem \ref{thm:param} with additional arguments. See Proposition \ref{prop:filtration} and Theorem \ref{thm:complete}. 


\brk 
Now an advantage of our categorical framework for studying finite-dimensional Hopf algebra actions on tensor algebras is that there are many finite (multi-)tensor categories $\mathcal{C}$ over which (indecomposable) semisimple $\mathcal{C}$-module categories are concretely understood, especially in terms of elementary group-theoretic data.  Our main result  [Theorem \ref{thm:param}] then allows us to classify tensor algebras in such finite tensor categories by elementary group-theoretic data.
For instance, due to results of Ostrik \cite{OstrikIMRN2003} and Natale \cite{Natale2016}, this is true for the pointed fusion category ${\sf Vec}_G^\omega$, whose objects are $G$-graded finite-dimensional vector spaces, for $G$ a finite group, where the associativity constraint is given by $\omega \in H^3(G, \kk^\times)$ [Proposition~\ref{prop:OstNat}]. Using this result, we study minimal ${\sf Vec}_G^\omega$-tensor algebras, via several detailed examples, in Section~\ref{sec:VecGw}.

\brk Indecomposable semisimple module categories over group-theoretical fusion categories $\mathcal{C}$ are also completely understood in terms of group-theoretic data, and we exploit this in Section~\ref{sec:tengt} to examine tensor algebras in such categories  $\mathcal{C}$, especially for those equipped with a fiber functor. We proceed  in Section~\ref{sec:recon} to describe indecomposable semisimple algebras in $\mathcal{C}$, then use this to explicitly classify indecomposable semisimple algebras in the representation category of the Kac-Paljutkin Hopf algebra $H_8$ in Section~\ref{sec:H8}.

\brk One could also apply Theorem~\ref{thm:param} to study tensor algebras in other finite tensor categories for which semisimple module categories are understood, and we leave this to future investigation. One could consider, for instance, the Drinfeld center of ${\sf Vec}_G^\omega$, and the category $\mathcal{C}_q$ consisting of comodules over the quantized function algebra $\mathcal{O}_q(SL_2)$; see  \cite[Theorem~3.6]{OstrikIMRN2003}, \cite[Theorem~2.1]{Ostrik:Nonss} \cite[Theorem~2.5]{EO}.

\brk Finally, in Section~\ref{sec:Path} we return to our study of finite-dimensional Hopf algebra actions on path algebras of quivers, by first introducing the notion of a {\it $\mathcal{C}$-path algebra} for a finite tensor category $\mathcal{C}$ equipped with a fiber functor $F$ [Definition~\ref{def:Cpath}]. This is simply a $\mathcal{C}$-tensor algebra $T_S(E)$ in the case when $F(S)$ is a commutative $\kk$-algebra; see Remark~\ref{rem:comm} for justification of this terminology. We determine  in this section necessary and sufficient conditions for a $\mathcal{C}$-tensor algebra to be a $\mathcal{C}$-path algebra, for $\mathcal{C}$ group-theoretical, and end with a discussion of ${\sf Rep}(H_8)$-path algebras in Example~\ref{ex:H8path}.


\section{Background material} \label{sec:bg}


In this section, we provide a review of certain monoidal categories, namely (multi-)tensor and (multi-) fusion categories. We also review module categories over and algebraic structures within these monoidal categories. We refer the reader to the text \cite{EGNO} and the references therein for further details.

\subsection{(Multi-)Tensor categories and (multi-)fusion categories}
Let $\C$ be a $\kk$-linear abelian category. We say that $\C$ is {\it locally finite} if each Hom space is a finite-dimensional $\kk$-vector space and if every object has finite length.  
Given such a $\C$, we write {\sf Ind}($\mathcal{C}$) for the ind-completion of $\mathcal{C}$.
 
\brk By $X \in \mathcal{C}$ we mean that $X$ is an object of $\mathcal{C}$.  A nonzero $X \in \mathcal{C}$ is {\it simple} (or {\it irreducible}) if {\bf 0} and $X$ are its only subobjects. A category $\mathcal{C}$ is {\it semisimple} if every object is a direct sum of simple objects. Moreover, $X \in \mathcal{C}$ is {\it projective} if the functor ${\sf Hom}_{\mathcal{C}}(X, $--$)$ is exact, and is {\it indecomposable} if it is nonzero and cannot be decomposed as the direct sum of nonzero subobjects. Simple objects in a category $\mathcal{C}$ are indecomposable, and the converse holds when $\mathcal{C}$ is semisimple. Let Irr($\mathcal{C}$) denote the set of isomorphism classes of simple objects of $\C$. We say that a locally finite $\kk$-linear abelian category $\mathcal{C}$ is {\it finite} if the cardinality of Irr($\mathcal{C}$) is finite and if $\mathcal{C}$ has enough projectives (i.e., every simple object has a projective cover); in this case, the cardinality of Irr($\mathcal{C}$) is referred to as  the {\it rank} of $\mathcal{C}$.

\brk A {\it multi-tensor category} $\mathcal{C}$ is a locally finite, $\kk$-linear, abelian, rigid, monoidal category, i.e., $\C$ is equipped with a bifunctor $\otimes \colon \mathcal{C} \times \mathcal{C} \to \mathcal{C}$, associativity isomorphisms $\{a_{X,Y,Z}\colon (X\otimes Y) \otimes Z \overset{\sim}{\to} X \otimes (Y \otimes Z) ~|~ X,Y,Z \in \mathcal{C}\}$, and a unit object ${\bf 1} \in \mathcal{C}$ with natural isomorphisms $l_X\colon \mathbf{1} \otimes X \xto{\sim} X$ and $r_X\colon X \otimes \mathbf{1} \xto{\sim} X$ for all $X \in \mathcal{C}$, and with isomorphism $\mathbf{1} \otimes \mathbf{1} \xto{\sim} \mathbf{1}$ all satisfying certain pentagon, triangle, and rigidity axioms \cite[Sections~2.1 and~2.10]{EGNO}. In particular, for any $X \in \C$ there exists objects $X^*$ and ${}^* X$ that serve as its {\it left dual} and {\it right dual}, respectively, via  {\it (co)evaluation} morphisms $({\sf co}){\sf ev}_X$ in $\C$.

\brk We have the following result on decomposing the identity object in a multi-tensor category.

\begin{lemma} \label{lem:idem}
In a multi-tensor category $ \mathcal{C}$, there is a decomposition ${\bf 1} = \bigoplus_{i}{\bf 1}_i$, where $\{{\bf 1}_i\}_i$ are pairwise non-isomorphic simple objects in $\mathcal{C}$ that satisfy ${\bf 1}_i \otimes {\bf 1}_j  \cong \delta_{ij} {\bf 1}_i$ and ${\bf 1}_i^* \cong {}^*{\bf 1}_i \cong {\bf 1}_i$. Furthermore, the set of summands $\{{\bf 1}_i\}_i$ are uniquely determined as subobjects of $\mathbf{1}$, up to reordering.
\end{lemma}
\begin{proof}
This is proved in \cite[Section~4.3]{EGNO} except that the uniqueness is not explicitly mentioned.  The summands $\mathbf{1}_i$ are the images of a complete collection of primitive orthogonal idempotents in $\mathsf{End}_{\C}(\mathbf{1})$.  Since $\mathsf{End}_{\C}(\mathbf{1})$ is semisimple, thus isomorphic to a direct sum of finitely many copies of $\kk$ \cite[Theorem~4.3.1]{EGNO}, these primitive idempotents are uniquely determined from $\mathbf{1}$.  Thus, the result holds. \end{proof}

\brk 
A multi-tensor category $\mathcal{C}$ is called a {\it tensor category}
if  {\sf End}$_{\mathcal{C}}({\bf 1}) \cong \kk$, that is, if {\bf 1} is a simple object of $\mathcal{C}$.  Examples of tensor categories include:
\begin{itemize}
\item the categories {\sf Rep}($G$) and {\sf Rep}($H$) of finite-dimensional representations of a finite group $G$ and of a finite-dimensional Hopf algebra $H$, respectively; and
\item the categories {\sf Vec}, ${\sf Vec}_G$,  and ${\sf Vec}_G^\omega$ of finite-dimensional $\kk$-vector spaces, those that are $G$-graded, and those that are $G$-graded with associativity constraint $\omega \in Z^3(G, \kk^\times)$, respectively.
\end{itemize}

\noindent Though not rigid, we also consider the category  $\bVec$ of all vector spaces.

  \brk A {\it fiber functor} on a tensor category $\mathcal{C}$ is an exact faithful tensor functor $F\colon \mathcal{C} \to {\sf Vec}$ such that $F({\bf 1}) = \kk$, equipped with natural isomorphisms $F(X) \otimes F(Y) \overset{\sim}{\to} F(X \otimes Y)$ for all $X,Y \in \mathcal{C}$. Examples of fiber functors include  forgetful functors on ${\sf Rep}(G)$, on ${\sf Rep}(H)$, and on ${\sf Vec}_G$. But there does not exist a fiber functor on ${\sf Vec}_G^\omega$ when $\omega$ is cohomologically nontrivial \cite[Example~5.1.3]{EGNO}.
  
\brk A tensor category $\mathcal{C}$ is  called {\it pointed} if all of its simple objects are invertible, i.e., $X \otimes X^* \cong X^* \otimes X \cong {\bf 1}$ via the co/evaluation maps for all $X \in \mathcal{C}$.
 Examples of invertible objects include 1-dimensional representations in  {\sf Rep}($G$) for  a group $G$. Hence, {\sf Rep}($G$) is a pointed tensor category when $G$ is an abelian group. The categories {\sf Vec}, ${\sf Vec}_G$,  and ${\sf Vec}_G^\omega$, for a finite  group $G$, are also pointed tensor categories \cite[Example~5.11.2]{EGNO}.

 \brk A {\it (multi-)fusion category} is a finite semisimple (multi-)tensor category $\mathcal{C}$.
  Examples of fusion categories include {\sf Rep}($H$) where $H$ is a semisimple  Hopf algebra, and also include  {\sf Vec}, ${\sf Vec}_G$,  and ${\sf Vec}_G^\omega$.

\brk Next, we turn our attention to module categories. A {\it left module category} over a multi-tensor category $\mathcal{C}$ is a locally finite, $\kk$-linear abelian category  $\mathcal{M}$ equipped with a bifunctor $\otimes\colon \mathcal{C} \times \mathcal{M} \to \mathcal{M}$ which is bilinear on morphisms and exact in the first variable, a natural isomorphism for associativity satisfying the pentagon axiom, and  for each $M \in \mathcal{M}$ a natural isomorphism ${\bf 1} \otimes M \overset{\sim}{\to} M$ satisfying the triangle axiom. {\it Right module categories} are defined analogously. A  module category $\mathcal{M}$ over $\mathcal{C}$ is {\it indecomposable} if it is nonzero and is not equivalent to a direct sum of two nontrivial module categories over $\mathcal{C}$. Moreover, $\mathcal{M}$ is called {\it faithful} if each simple object ${\bf 1}_i \in \mathcal{C}$ in Lemma~\ref{lem:idem} acts by a nonzero functor on $\mathcal{M}$. Also, $\mathcal{M}$ is {\it exact} if for every projective object $P \in \mathcal{C}$ and every object $M \in \mathcal{M}$ the object $P \otimes M$ is projective in $\mathcal{M}$. Note that any semisimple module category is exact since any object in a semisimple category is projective. Furthermore, if $\C$ is multi-fusion, then the exact module categories over $\C$ are precisely the semisimple ones.

\brk The collection of (left/ right) module categories over a multi-tensor category $\mathcal{C}$ forms a 2-category, which is denoted by ${\sf Mod}(\mathcal{C})$ \cite[Remark~7.12.15]{EGNO}.
For  $\mathcal{C}$-module categories $\mathcal{M}$ and $\mathcal{N}$, denote by {\sf Fun}$_{\mathcal{C}}(\mathcal{M},\mathcal{N})$ the category consisting of right exact $\mathcal{C}$-module functors $ \mathcal{M} \to \mathcal{N}$. We denote the category {\sf Fun}$_{\mathcal{C}}(\mathcal{M},\mathcal{M})$ by $\mathcal{C}_{\mathcal{M}}^*$ and call it the {\it dual category to $\mathcal{C}$ with respect to $\mathcal{M}$}.

\begin{theorem} \label{thm:Fun-ss} \cite[Sections~7.11, 7.12]{EGNO} \cite[Theorems~2.15, 2.18]{ENO}
For any exact (resp., finite semisimple) module categories $\mathcal{M}, \mathcal{N}$ over a finite multi-tensor (resp., multi-fusion) category $\mathcal{C}$, we have that the category {\sf Fun}$_{\mathcal{C}}(\mathcal{M},\mathcal{N})$ is finite (resp., finite semisimple). Moreover, the category $\mathcal{C}_{\mathcal{M}}^*$  is a finite tensor (resp., fusion) category when $\mathcal{C}$ is a finite tensor (resp., fusion) category and $\mathcal{M}$ is indecomposable. \qed
\end{theorem}

Two (multi-)tensor categories $\mathcal{C}$ and $\mathcal{D}$ are {\it categorically Morita equivalent} if there exists an exact $\mathcal{C}$-module category $\mathcal{M}$ so that $\mathcal{D}^{\rm op}$ is tensor equivalent to $\mathcal{C}_{\mathcal{M}}^*$.  The following result shows that there is a bijection between module categories over $\mathcal{C}$ and those over the dual $\mathcal{C}_{\mathcal{M}}^*$ with respect to a faithful, exact module category $\mathcal{M}$. 

\begin{theorem} \label{thm:2categ} \cite[Theorem~7.12.16]{EGNO}
Let $\mathcal{M}$ be a faithful, exact  module category over a multi-tensor category $\mathcal{C}$. Then the 2-functor ${\sf Mod}(\mathcal{C}) \to {\sf Mod}((\mathcal{C}^*_\mathcal{M})^{\rm op})$ given by $\mathcal{N} \mapsto {\sf Fun}_{\mathcal{C}}(\mathcal{M}, \mathcal{N})$ is a 2-equivalence. \qed
\end{theorem}

In other words, there is a natural bijection between exact  module categories over two categorically Morita equivalent multi-tensor  categories.


\subsection{Algebras, ideals, and (bi)modules in finite multi-tensor categories}\label{sec:algebrasetc}
Let  $\mathcal{C}:=(\mathcal{C}, \otimes, a, l, r, {\bf 1})$ be a multi-tensor category. 

\brk An {\it algebra} in $\mathcal{C}$ is a triple $(A,m,u)$, where $A \in \mathcal{C}$, and $m\colon A \otimes A \to A$ (multiplication) and $u\colon {\bf 1} \to A$ (unit) are morphisms in $\C$ that are compatible with the associativity constraint $a$ and the unit constraints $l,r$, respectively. 
 A {\it right module} over an algebra $(A,m,u)$ in $\mathcal{C}$ is a pair $(M, \rho^{rt})$, where $M$ is an object of $\mathcal{C}$ and $\rho^{rt}\colon M\otimes A \to M$ is a morphism in $\C$ that is compatible with the associativity constraint $a$ and the unit constraints $l,r$. A {\it left module} $(M, \rho^{lt})$ over an algebra $(A,m,u)$ is defined likewise. Let ${\sf Mod}_{\mathcal{C}}$-$A$ be the category of right modules over $A$; this is a left module category over $\mathcal{C}$.
Moreover, two algebras $A$ and $B$ in $\mathcal{C}$ are called {\it Morita equivalent} if ${\sf Mod}_{\mathcal{C}}$-$A$ and ${\sf Mod}_{\mathcal{C}}$-$B$ are equivalent as left $\mathcal{C}$-module categories.

\brk We say that an algebra $A$ in $\mathcal{C}$ is {\it semisimple} (resp., {\it indecomposable}, {\it exact}) if the category {\sf Mod}$_{\mathcal{C}}$-$A$ is a semisimple (resp., indecomposable, exact) category.

\begin{theorem} \label{thm:ModcA} \cite[Corollary~7.10.5]{EGNO} \cite[Theorem 3.1]{OstrikTG2003} Given a finite multi-tensor (resp., multi-fusion) category $\mathcal{C}$,
 each exact (resp., finite semisimple) module category $\mathcal{M}$ over $\mathcal{C}$ is equivalent to {\sf Mod}$_{\mathcal{C}}$-$A$ for some exact (resp., semisimple) algebra $A \in \mathcal{C}$. \qed
\end{theorem}

\begin{example} \label{ex:S=1} \cite[Examples~7.8.4,~7.8.11,~7.8.18,~7.10.2]{EGNO}
Consider the algebra $A = {\bf 1}$ in $\mathcal{C}$. Since {\sf Mod}$_\mathcal{C}$-${\bf 1}$ is equivalent to $\mathcal{C}$ as $\mathcal{C}$-module categories, we have that the algebra ${\bf 1}$ is exact (resp., semisimple) when $\mathcal{C}$ is finite (resp., finite semisimple),  and it is indecomposable when $\C$ is a tensor category. In this case, the algebra {\bf 1} is Morita equivalent to the algebra $X \otimes X^*$ for any nonzero $X \in \mathcal{C}$, where the latter has multiplication $m = {\sf id}_X \otimes {\sf ev}_X \otimes {\sf id}_{X^*}$ and unit $u = {\sf coev}_X$. Thus, {\sf Mod}$_\mathcal{C}$-$(X \otimes X^*)$ is equivalent to $\mathcal{C}$ as $\mathcal{C}$-module categories as well, for any $X \in \mathcal{C}$.
\end{example}

For instance, if $\mathcal{C} = {\sf Vec}$ in the example above, then ${\bf 1} = \kk$. Moreover, for $X$ an $n$-dimensional vector space, the algebra $X \otimes X^*$ is isomorphic to the matrix algebra ${\sf Mat}_n(\kk)$, which is well-known to be Morita equivalent to $\Bbbk$ in ${\sf Vec}$.

\brk Now we discuss bimodules in $\C$. Let $A,B$ be two algebras in $\mathcal{C}$. An $(A,B)$-{\it bimodule} in $\mathcal{C}$ is a triple $(M, \rho^{lt}, \rho^{rt})$, where $M \in \mathcal{C}$ and $\rho^{lt}\colon A \otimes M\to M$ and $\rho^{rt}\colon M\otimes B \to M$ are morphisms in $\mathcal{C}$ so that $(M, \rho^{lt})$ is a left $A$-module in $\mathcal{C}$, $(M, \rho^{rt})$ is a right $B$-module in $\mathcal{C}$, and $\rho^{lt} \circ (\text{id}_A \otimes \rho^{rt}) \circ a_{A,M,B} = \rho^{rt} \circ (\rho^{lt} \otimes \text{id}_B)$ as morphisms from $A \otimes M \otimes B$ to $M$ in $\C$. 
We write ${\sf Bimod}_\mathcal{C}(A,B)$ (resp., {\sf Bimod}$_{\mathcal{C}}(A)$) for the category of $(A,B)$-bimodules (resp., $(A,A)$-bimodules) in $\mathcal{C}$.

\brk Take two $\C$-module categories $\mathcal{M} \sim {\sf Mod}_\C$-$A$ and $\mathcal{N} \sim {\sf Mod}_\C$-$B$. Then, ${\sf Bimod}_\C(A,B)$ is equivalent to ${\sf Fun}_\C(\mathcal{M}, \mathcal{N})$ \cite[Proposition~7.11.1]{EGNO}.  Moreover, $\mathcal{C}_{\mathcal{M}}^*$ is tensor equivalent to ${\sf Bimod}_{\mathcal{C}}(A)^{\rm op}$ \cite[Remark~7.12.5]{EGNO}.  We also have by Theorem~\ref{thm:Fun-ss} the  result below.

\begin{proposition} \label{prop:bimod} \cite{ENO}
Given an exact (resp., semisimple) algebra $A$ in a finite multi-tensor (resp., multi-fusion) category $\mathcal{C}$, we have that ${\sf Bimod}_{\mathcal{C}}(A)$ is a finite multi-tensor (resp., multi-fusion) category with unit object $A$, and it is a finite tensor (resp., fusion) category when $\mathcal{C}$ is a finite tensor (resp., fusion) category and $A$ is indecomposable. \qed
\end{proposition}

We get the following consequence.

\begin{corollary} \label{cor:Adecomp}
Given an exact (resp., semisimple) algebra $A$ in a finite multi-tensor category (resp., multi-fusion) $\mathcal{C}$, we can uniquely decompose the algebra $A \in {\sf Bimod}_\C(A)$ into a direct sum of indecomposable subalgebras $\{A_i\}_{i}$ in ${\sf Bimod}_\C(A)$, with $A_i$ being pairwise non-isomorphic $A$-bimodules and $A_i \otimes_{A} A_j \cong \delta_{ij} A_i$. 
\end{corollary}

\begin{proof}
Consider $A$ as ${\bf 1}_{{\sf Bimod}_\mathcal{C}(A)}$. Now by Proposition~\ref{prop:bimod} and Lemma~\ref{lem:idem}, we have that $A$ is uniquely a direct sum of indecomposable, pairwise non-isomorphic objects $A_i$ in ${\sf Bimod}_\mathcal{C}(A)$ so that $A_i \otimes_{A} A_j \cong \delta_{ij} A_i$. 
\end{proof}

\brk Finally,  an {\it ideal} $I$ of an algebra $A \in \mathcal{C}$ is an $A$-$A$-sub-bimodule of $A$ in $\C$, and with this one can form the quotient algebra $A/I$ in $\C$. Moreover, an ideal $I$ of $A \in \mathcal{C}$ can be realized as a subalgebra of $A \in {\sf Bimod}_\mathcal{C}(A)$ if $I$ is a direct summand of $A$ in  ${\sf Bimod}_\mathcal{C}(A)$; in this case, the unit $A \to I$ is given by projection onto $I$.


\section{Main results} \label{sec:main}

Before providing a tensor-categorical framework for studying finite quantum symmetries of tensor algebras, we review the connection to quivers that motivated this work.

\brk Recall that a {\it quiver} is another name for a directed graph, in the context where the directed graph is used to define an algebra.  Formally, a (finite) quiver $Q=(Q_0, Q_1, s, t)$ consists of a (finite) set of {\it vertices} $Q_0$, a (finite) set of {\it arrows} $Q_1$, and two functions $s, t\colon Q_1 \to Q_0$ giving the {\it source} and {\it target} of each arrow.
We assume all quivers in this paper are finite.
 One can construct a {\it path algebra} $\kk Q$ from a quiver $Q$ which is a $\kk$-algebra whose basis consists of all paths in $Q$, with multiplication of basis elements given by concatenation of paths whenever it is defined and 0 otherwise.  Such an algebra is naturally graded by path length. 
Moreover, the path algebra $\kk Q_0$ is taken to be the path algebra of the quiver $(Q_0, \emptyset)$ with no arrows; thus, $\kk Q_0$ is a commutative semisimple (so, exact) $\kk$-algebra. 

\brk 
Path algebras arise as a special case of the following construction.
Given a finite dimensional semisimple $\kk$-algebra $B$, and a $B$-bimodule $V$, we can construct the {\it tensor algebra} $T_B(V) = \bigoplus_{i \geq 0} V^{\otimes_B i}$, where \red{$V^{\otimes_B 0} = B$}.
It is an object of ${\sf Bimod}_{\bVec}(B)$ which admits an $\mathbb{N}$-grading via $T_B(V)_n = V^{\otimes_B n}$ for $n \in \mathbb{N}$.
When $B$ and $V$ both lie in $\sVec$, the tensor algebra $T_B(V)$ is {\it finitely generated as a $\kk$-algebra} and it is then an object of ${\sf Ind}({\sf Bimod}_{{\sf Vec}}(B))$.  Observe that the path algebra $\kk Q$ is isomorphic to the finitely generated tensor algebra $T_{\kk Q_0}(\kk Q_1)$, since $\kk Q_1$ is naturally a $\kk Q_0$-bimodule. 
Conversely, we have the following classical theorem.
\begin{theorem}[P. Gabriel, see \textnormal{\cite[Ch.~2]{ASS2006}}] \label{thm:eqKQ}
Every finitely generated tensor algebra $T_B(V)$ as above is Morita equivalent to the path algebra of a quiver. \qed
\end{theorem}

We move beyond {\sf Vec} to finite multi-tensor categories as follows.
In Section~\ref{sec:Cmulti}, we discuss tensor algebras $T$ with base algebras and generating bimodules in finite multi-tensor categories $\C$; these are called {\it $\C$-tensor algebras} [Definition~\ref{S,E}].  Our first result is that such $T$ can be decomposed uniquely as a collection of {\it minimal} ones in the sense of Definition~\ref{def:minl} [Proposition~\ref{prop:decomp}]. In Section~\ref{sec:minimal}, we  establish our main result on classifying equivalence classes of $\C$-tensor algebras in Theorem~\ref{thm:param}. Finally in Section~\ref{sec:CHopf} we take $\C = {\sf Rep}(H)$, for $H$ a semisimple Hopf algebra, and show that any $H$-action on a tensor algebra that preserves its ascending filtration must be isomorphic to a grade-preserving action [Proposition~\ref{prop:filtration}]; we then prove a similar result for $H$-actions on completed tensor algebras that preserve the natural descending filtration [Theorem~\ref{thm:complete}].


\subsection{Tensor algebras in finite multi-tensor categories} \label{sec:Cmulti}

Let $\mathcal{C}$ be a finite multi-tensor category. We introduce the notion of a $\C$-tensor algebra as follows.

\begin{definition}[$S$, $E$, $T_S(E)$] \label{S,E}
Fix $S$ an exact  algebra in $\C$, and fix an $S$-bimodule $E$ in ${\sf Ind}(\mathcal{C})$. 

\begin{enumerate}
\item Form the algebra $T_S(E)$ in  ${\sf Ind}(\mathcal{C})$, or more specifically in ${\sf Bimod}_{{\sf Ind}(\mathcal{C})}(S)$, given by 
$$T_S(E) = S \oplus E \oplus (E \otimes_S E) \oplus (E \otimes_S E \otimes_S E) \oplus \cdots,$$
with multiplication morphism given by the natural maps $(E^{\otimes_S n}) \otimes_S (E^{\otimes_S m}) \to E^{\otimes_S (n+m)}$, 
and with unit morphism induced from the unit of $S$ by $S \hookrightarrow T_S(E)$. We call this a {\it $\mathcal{C}$-tensor algebra}. 

\smallskip

\item We refer to $S$ and $E$ as the {\it base algebra} and {\it generating bimodule} of  $T_S(E)$, respectively. 

\smallskip

\item  If $E$ also belongs to $\C$, then we say that $T_S(E)$ is {\it finitely generated (f.g.).}
\end{enumerate}
\end{definition}

From now on, we concentrate on f.g. $\C$-tensor algebras. In this case, note that $E$ has finite length as an $S$-bimodule in $\mathcal{C}$ by Proposition~\ref{prop:bimod}.
Further, $T_S(E)$ admits a natural $\mathbb{N}$-grading with $(T_S(E))_n = E^{\otimes_S n} \in \mathcal{C}$ for $n \in \mathbb{N}$, and is an object of ${\sf Ind}({\sf Bimod}_\C(S))$.

\begin{example}\label{ex:HQ}
Let $Q$ be a  finite quiver and $\kk Q$ be its path algebra.  Suppose we have an action of a finite-dimensional Hopf algebra $H$ on the algebra $\kk Q$ preserving the grading by path length.
Taking $\C=\Rep(H)$, the setup above makes $S=\kk Q_0$ an exact algebra in $\C$ and $E=\kk Q_1$ an $S$-bimodule in $\C$. The identification $\kk Q \cong T_S(E)$ described above for $\kk$-algebras is an isomorphism of graded algebras in ${\sf Ind}(\C)$.  
For example, this could arise from a finite group $G$ acting by directed graph automorphisms of $Q$ and by extending linearly to a $\kk G$-action on $\kk Q$. 
For $H = \kk G$, see related works of Reiten-Riedtmann \cite{RR} and of Demonet \cite{Demonet}. 
\end{example}

Next, we introduce the notion of equivalence for $\C$-tensor algebras.
Consider $\C$-tensor algebras $T_S(E)$ and $T_{S'}(E')$ and set $\mathcal{M}:= {\sf Mod}_\C\text{-}S$ and $\mathcal{M}':= {\sf Mod}_\C$-$S'$.

\begin{definition} \label{def:equiv}
An \emph{equivalence} from $T_S(E)$ to $T_{S'}(E')$ is a pair $(F, \xi)$ where 
\begin{itemize}
\item $F\colon \mathcal{M} \to \mathcal{M}'$ is an equivalence of $\C$-module categories, realized by $F(M) = M \otimes_S X$ where \linebreak $X=F(S) \in \Bimod_\C(S, S')$;
\item $\xi \colon \bar{X} \otimes_S E \otimes_S X \xto{\sim} E'$ is an isomorphism of $S'$-bimodules, where $\bar{X}=F^{-1}(S') \in \Bimod_\C(S', S)$ realizes the $\C$-module functor $F^{-1}$.
\end{itemize}
\end{definition}

Note that any $\C$-module equivalence $F: \mathcal{M} \to \mathcal{M}'$ defines a tensor equivalence $\widehat{F}: \C_\mathcal{M}^* \to \C_{\mathcal{M}'}^*$, and $\xi$ is just an isomorphism $\widehat{F}(E) \xto{\sim} E'$. On the other hand, a simple example illustrating equivalence of tensor algebras defined by nonisomorphic generating bimodules is found in Example \ref{ex:quiverminimal}, and more subtle examples throughout Section \ref{sec:VecGw}. Next, we show that equivalence of tensor algebras is sufficient to imply Morita equivalence as algebras in~$\C$.

\begin{proposition}
An equivalence $(F,\xi)$ as in Definition \ref{def:equiv} induces an equivalence of $\C$-module categories from $\Mod_\C$-$T_S(E)$ to $\Mod_\C$-$T_{S'}(E')$.
\end{proposition}

\begin{proof}
Retain the notation of Definition~\ref{def:equiv}. The data defining a right $T_S(E)$-module is equivalent to giving (i) a right $S$-module $Y$, and (ii) a morphism of right $S$-modules  $\phi\colon Y \otimes_S E\to Y$.  The corresponding right $T_{S'}(E')$-module is given by (i) the right $S'$-module $F(Y) = Y \otimes_S X$ and (ii) the morphism of right $S'$-modules $\phi'\colon (Y \otimes_S X)\otimes_{S'} E'\to Y \otimes_S X$ given by the composition (omitting associativity and unit isomorphisms)
\[
(Y \otimes_S X)\otimes_{S'} E' 
\xto{\id_{Y\otimes_S X} \otimes \xi^{-1}} (Y \otimes_S X) \otimes_{S'} \bar{X}\otimes_S E \otimes_S X 
\cong (Y \otimes_S E) \otimes_S X \xto{\phi \otimes \id_X} Y \otimes_S X.  \qedhere
\]
\end{proof}

Our first result is  a unique decomposition theorem for  f.g. $\C$-tensor algebras $T_S(E)$ by writing them as a combination of {\it minimal} components, as defined below.

\begin{definition} \label{def:minl}
Take an exact algebra $S \in \mathcal{C}$ and $E \in {\sf Bimod}_{\mathcal{C}}(S)$, and  a f.g.  $\C$-tensor algebra $T_S(E)$.
\begin{enumerate}
\item We call  $T_S(E)$ {\it minimal} when $E$ is an indecomposable $S$-bimodule in $\mathcal{C}$.
\smallskip
\item If $E = \oplus_i E_i$, for $E_i \in {\sf Bimod}_\C(S)$  indecomposable, then the f.g. $\C$-tensor algebras $\{T_S(E_i)\}_i$ are called the \emph{minimal components} of $T_S(E)$.
\end{enumerate}
\end{definition}

The result below is immediate from the Krull-Schmidt theorem applied to $E$ \cite[Theorem~1.5.7]{EGNO}.

\begin{proposition}\label{prop:decomp}
Let $T_S(E)$ be  a f.g. $\C$-tensor algebra. Then the minimal components $\{T_S(E_i)\}_i$ of $T_S(E)$ are uniquely determined up to reordering and isomorphism class of $E_i$. Moreover, $T_S(E) = T_S(\oplus_i E_i)$ can be reconstructed as a free product of its minimal components $\{T_S(E_i)\}_i$. \qed
\end{proposition}

\brk Having fixed a minimal component, it will often be possible to simplify the base algebra, as illustrated in the remark below.

\begin{remark}\label{rem:nonunique}
Suppose we have a  f.g. $\C$-tensor algebra $T_S(E)$ and $S$ decomposes as $S=S'\oplus S''$ as an algebra in $\C$.  When $S'$ acts trivially on $E$ \red{(even for $T_S(E)$ minimal)} so that $T_S(E) = S' \oplus T_{S''}(E)$ as algebras, one can study the smaller algebra $T_{S''}(E)$ in place of $T_S(E)$, particularly for classification purposes.  More generally, if $S$ has an ideal $I$  that acts trivially on both sides of $E$, then we could replace $S$ with $S/I$.  
\end{remark}

We keep $S$ fixed in the definition of ``minimal components'' in order to reconstruct $T_S(E)$ from its minimal components uniquely, prompting the terminology below.

\begin{definition}\label{def:S-faithful}
 We say that a f.g. $\mathcal{C}$-tensor algebra $T_S(E)$ is \emph{($S$-)faithful} if there does not exist a nonzero ideal $I$ of $S$ in $\C$ so that $E \in {\sf Bimod}_{\C}(S/I)$. 
 \end{definition}

Basic examples illustrating the minimal and faithful properties (or lack thereof) are shown in Figure \ref{fig:minfaith}.
 
 
\begin{figure}
\[
\begin{array}{ccccccc}
\vcenter{\hbox{
 \begin{tikzpicture}[point/.style={shape=circle,fill=black,scale=.5pt,outer sep=3pt},>=latex]
   \node[point] (1)  at (0.4,0) {};
   \node[point] (2) at (-0.4,0) {};
      \node[point] (3)  at (0.4,-.7) {};
   \node[point] (4) at (-0.4,-.7) {};
  \draw[<->,dotted,red] (1) to   (2);
    \draw[<->,dotted,red] (3) to   (4);
\draw[->]  (1) to [out=60,in=300,looseness=8]  (1);
 \draw[->]  (2) to [out=120,in=240,looseness=8]  (2);
\draw[->]  (1) to [out=70,in=290,looseness=14]  (1);
 \draw[->]  (2) to [out=110,in=250,looseness=14]  (2);
   \end{tikzpicture}}}
&&
\vcenter{\hbox{\begin{tikzpicture}[point/.style={shape=circle,fill=black,scale=.5pt,outer sep=3pt},>=latex]
   \node[point] (1)  at (0.4,0) {};
   \node[point] (2) at (-0.4,0) {};
      \node[point] (3)  at (0.4,-.7) {};
   \node[point] (4) at (-0.4,-.7) {};
  \draw[<->,dotted,red] (1) to   (2);
    \draw[<->,dotted,red] (3) to   (4);
\draw[->]  (1) to [out=60,in=300,looseness=8]  (1);
 \draw[->]  (2) to [out=120,in=240,looseness=8]  (2);
   \end{tikzpicture}}}
&&
\vcenter{\hbox{\begin{tikzpicture}[point/.style={shape=circle,fill=black,scale=.5pt,outer sep=3pt},>=latex]
   \node[point] (1)  at (0.4,0) {};
   \node[point] (2) at (-0.4,0) {};
      \node[point] (3)  at (0.4,-.7) {};
   \node[point] (4) at (-0.4,-.7) {};
  \draw[<->,dotted,red] (1) to   (2);
    \draw[<->,dotted,red] (3) to   (4);
\draw[->]  (1) edge  (3);
\draw[->]  (2) edge  (4);
\draw[->]  (1) to [out=60,in=300,looseness=8]  (1);
 \draw[->]  (2) to [out=120,in=240,looseness=8]  (2);
   \end{tikzpicture}}}
&&
\vcenter{\hbox{\begin{tikzpicture}[point/.style={shape=circle,fill=black,scale=.5pt,outer sep=3pt},>=latex]
   \node[point] (1)  at (0.4,0) {};
   \node[point] (2) at (-0.4,0) {};
      \node[point] (3)  at (0.4,-.7) {};
   \node[point] (4) at (-0.4,-.7) {};
  \draw[<->,dotted,red] (1) to   (2);
    \draw[<->,dotted,red] (3) to   (4);
\draw[->]  (1) edge  (3);
\draw[->]  (2) edge  (4);
 \draw[->, white, opacity=0]  (1) to [out=60,in=300,looseness=8]  (1);
   \end{tikzpicture}}}
   \vspace{-.1in}
\\
\text{not minimal, not faithful} && \text{minimal, not faithful} && \text{faithful, not minimal} && \text{minimal and faithful}\\
\end{array} 
\]
\centering
\captionsetup{justification=centering}
\caption{ Quivers with actions of $G = \mathbb{Z}_2$ (indicated by the dotted red arrow) that give rise to path algebras in $\mathsf{Rep}(G)$, or $\mathsf{Rep}(G)$-tensor algebras, with the properties listed above}\label{fig:minfaith}
\end{figure}
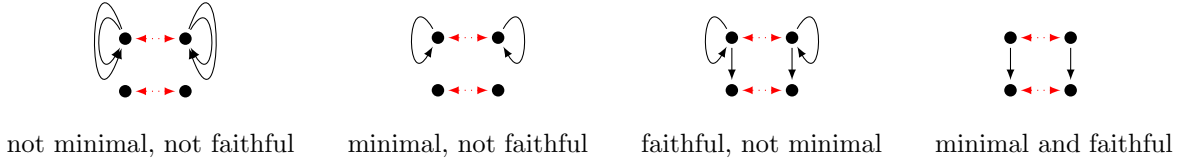

We end this part by illustrating the failure of reconstructing $T_S(E)$ from the indecomposable summands of $E$, if we were to further reduce the base algebras of minimal components of $T_S(E)$ that are not $S$-faithful.

\begin{example}
Let $\C=\Rep(G)$ where $G=\langle g \rangle$ is cyclic of order 2, and consider the actions of $G$ on the following two quivers indicated by the dotted red arrow.

\vspace{-.25in}

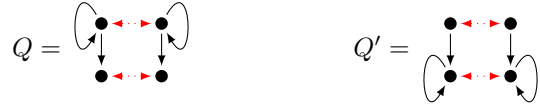
\[
Q= \vcenter{\hbox{\begin{tikzpicture}[point/.style={shape=circle,fill=black,scale=.5pt,outer sep=3pt},>=latex]
   \node[point] (1)  at (0.4,0) {};
   \node[point] (2) at (-0.4,0) {};
      \node[point] (3)  at (0.4,-.7) {};
   \node[point] (4) at (-0.4,-.7) {};
  \draw[<->,dotted,red] (1) to   (2);
    \draw[<->,dotted,red] (3) to   (4);
\draw[->]  (1) edge  (3);
\draw[->]  (2) edge  (4);
\draw[->]  (1) to [out=60,in=300,looseness=8]  (1);
 \draw[->]  (2) to [out=120,in=240,looseness=8]  (2);
 \draw[->, white, opacity=0]  (3) to [out=60,in=300,looseness=8]  (3);
   \end{tikzpicture}}}
\qquad \qquad \qquad
Q'=
    \vcenter{\hbox{\begin{tikzpicture}[point/.style={shape=circle,fill=black,scale=.5pt,outer sep=3pt},>=latex]
   \node[point] (1)  at (0.4,0) {};
   \node[point] (2) at (-0.4,0) {};
      \node[point] (3)  at (0.4,-.7) {};
   \node[point] (4) at (-0.4,-.7) {};
  \draw[<->,dotted,red] (1) to   (2);
    \draw[<->,dotted,red] (3) to   (4);
\draw[->]  (1) edge  (3);
\draw[->]  (2) edge  (4);
 \draw[->, white, opacity=0]  (1) to [out=60,in=300,looseness=8]  (1);
\draw[->]  (3) to [out=60,in=300,looseness=8]  (3);
 \draw[->]  (4) to [out=120,in=240,looseness=8]  (4);
   \end{tikzpicture}}}
\]

\vspace{-.1in}

\noindent This gives rise to two path algebras $\kk Q$ and $\kk Q'$ with actions of $G$, and thus two f.g. $\C$-tensor algebras as in Example \ref{ex:HQ}; call them $T_S(E)$ and $T_{S}(E')$. The minimal components of $T_S(E)$ and $T_{S}(E')$ are the path algebras on the following two sets of quivers, respectively.

\vspace{-.2in}

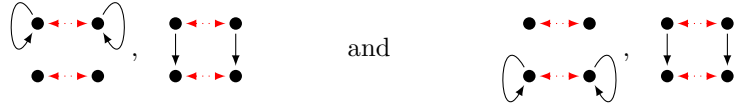
\[
\vcenter{\hbox{\begin{tikzpicture}[point/.style={shape=circle,fill=black,scale=.5pt,outer sep=3pt},>=latex]
   \node[point] (1)  at (0.4,0) {};
   \node[point] (2) at (-0.4,0) {};
      \node[point] (3)  at (0.4,-.7) {};
   \node[point] (4) at (-0.4,-.7) {};
  \draw[<->,dotted,red] (1) to   (2);
    \draw[<->,dotted,red] (3) to   (4);
\draw[->]  (1) to [out=60,in=300,looseness=8]  (1);
 \draw[->]  (2) to [out=120,in=240,looseness=8]  (2);
  \draw[->, white, opacity=0]  (3) to [out=60,in=300,looseness=8]  (3);
   \end{tikzpicture}}}, \quad
\vcenter{\hbox{\begin{tikzpicture}[point/.style={shape=circle,fill=black,scale=.5pt,outer sep=3pt},>=latex]
   \node[point] (1)  at (0.4,0) {};
   \node[point] (2) at (-0.4,0) {};
      \node[point] (3)  at (0.4,-.7) {};
   \node[point] (4) at (-0.4,-.7) {};
  \draw[<->,dotted,red] (1) to   (2);
    \draw[<->,dotted,red] (3) to   (4);
\draw[->]  (1) edge  (3);
\draw[->]  (2) edge  (4);
   \end{tikzpicture}}}
\qquad\qquad \text{and} \qquad \qquad
    \vcenter{\hbox{\begin{tikzpicture}[point/.style={shape=circle,fill=black,scale=.5pt,outer sep=3pt},>=latex]
   \node[point] (1)  at (0.4,0) {};
   \node[point] (2) at (-0.4,0) {};
      \node[point] (3)  at (0.4,-.7) {};
   \node[point] (4) at (-0.4,-.7) {};
  \draw[<->,dotted,red] (1) to   (2);
    \draw[<->,dotted,red] (3) to   (4);
     \draw[->, white, opacity=0]  (1) to [out=60,in=300,looseness=8]  (1);
\draw[->]  (3) to [out=60,in=300,looseness=8]  (3);
 \draw[->]  (4) to [out=120,in=240,looseness=8]  (4);
   \end{tikzpicture}}}, \quad
       \vcenter{\hbox{\begin{tikzpicture}[point/.style={shape=circle,fill=black,scale=.5pt,outer sep=3pt},>=latex]
   \node[point] (1)  at (0.4,0) {};
   \node[point] (2) at (-0.4,0) {};
      \node[point] (3)  at (0.4,-.7) {};
   \node[point] (4) at (-0.4,-.7) {};
  \draw[<->,dotted,red] (1) to   (2);
    \draw[<->,dotted,red] (3) to   (4);
\draw[->]  (1) edge  (3);
\draw[->]  (2) edge  (4);
   \end{tikzpicture}}} 
\]

\vspace{-.1in}

If, in each minimal component, we removed the summands of $S$ acting trivially on that component, we would arrive at the same set of underlying quivers for  both $T_S(E)$ and $T_S(E')$.
In this case, we would be unable to uniquely reconstruct the original tensor algebras $T_S(E)$ and $T_{S'}(E')$ from  their minimal components, since $T_S(E) \not \cong T_{S}(E')$ as algebras in ${\sf Ind}(\C)$ (i.e., they are not even isomorphic as $\kk$-algebras after forgetting the action of $G$ \cite{LLX}).
\end{example}


\subsection{Classification of tensor algebras}\label{sec:minimal}
When $\C$ is a  finite multi-tensor  category, the classification of minimal (f.g.) $\C$-tensor algebras can be carried out in terms of module categories over $\C$ as we see below. Recall our notion of equivalence for \C-tensor algebras in Definition \ref{def:equiv}.

\begin{theorem} \label{thm:param} 
Let $\C$ be a finite multi-tensor category.
Equivalence classes of (f.g.) $\mathcal{C}$-tensor algebras are in bijection with the pairs  $(\mathcal{M}, [U])$, where
\begin{enumerate}
\item[\textnormal{(i)}] $\mathcal{M}$ is an exact  $\mathcal{C}$-module category; and
\item[\textnormal{(ii)}]  $[U]$ is the conjugacy class of an object $U$ in ${\sf Fun}_{\mathcal{C}}(\mathcal{M}, \mathcal{M})$ $= \C_\M^*$.
\end{enumerate}
Here, $U, U' \in \C_\M^*$ are conjugate if there exists an autoequivalence (i.e., an invertible object) $F \in \C_\M^*$ so that $F U F^{-1}$ is isomorphic to $U'$.

Furthermore, the equivalence classes of minimal (f.g.) $\C$-tensor algebras are classified by the above data when restricting to $U$ indecomposable.
\end{theorem}

\begin{proof} 
For any exact algebra $S$ in $\C$ we have that $\Mod_\C$-$S$ is an exact $\C$-module category, by definition. On the other hand, by Theorem~\ref{thm:ModcA}, every exact $\mathcal{C}$-module category is equivalent to one of the form $\Mod_{\mathcal{C}}$-$S$ for some exact algebra $S$ in $\mathcal{C}$.
This shows that the choice for the base algebra $S$ of a $\C$-tensor algebra, up to Morita equivalence of $S$, is determined by the data in (i), up to equivalence of $\C$-module categories.

Fixing such an $S$, it is immediate from Definition \ref{def:equiv} that equivalence classes of $\C$-tensor algebras $T_S(E)$ are in bijection with conjugacy classes of objects $U$, and minimal tensor algebras correspond to indecomposable $U$ by definition.
\end{proof}

Let us consider in more detail the case of minimal $\C$-tensor algebras.  Classifying $\C$-tensor algebras can be carried out by first classifying base algebras $S$ via exact $\C$-module categories. Then, one needs to classify indecomposable $S$-bimodules $E$, identifying those which are invertible in $\Bimod_\C(S)$, and then compute conjugates $\bar{X} \otimes_S E \otimes_S X$ (where $X$ varies over the invertible bimodules and $\bar{X}$ denotes the inverse of $X$ in the bimodule category here and below) to identify a representative from each conjugacy class.

When $S=S_1\oplus S_2$ with each $S_i$ an indecomposable algebra in $\C$, 
an $S$-bimodule can be identified with a $2 \times 2$ matrix $E=(E_{ij})$ where each $E_{ij}$ is an $(S_i, S_j)$-bimodule in $\C$.  Restricting to minimal, faithful $\C$-tensor algebras, we can assume
\[
E= \begin{pmatrix} 0 & E_{12} \\ 0 & 0 \end{pmatrix}, \qquad E_{12} \in \Bimod_{\C}(S_1, S_2) \quad \text{indecomposable.}
\]
Without loss of generality we may consider invertible $S$-bimodules of the form
\begin{equation*}
X= \begin{pmatrix} X_1 & 0 \\ 0 & X_2 \end{pmatrix}, \qquad X_i \text{ an invertible $S_i$-bimodule.}
\end{equation*}
Then $\otimes_S$ corresponds to matrix multiplication and we compute the conjugation as
\begin{equation} \label{eq:inv}
\bar{X} \otimes_S E \otimes _S X= \begin{pmatrix} \bar{X_1} & 0 \\ 0 & \bar{X_2} \end{pmatrix}
\begin{pmatrix} 0 & E_{12} \\ 0 & 0 \end{pmatrix}
\begin{pmatrix} X_1 & 0 \\ 0 & X_2 \end{pmatrix}
=\begin{pmatrix} 0 & \bar{X_1}  \otimes_S E_{12}\otimes_S X_2 \\ 0 & 0 \end{pmatrix}.
\end{equation}
A particular consequence is recorded in the following \red{proposition}.

\begin{proposition}\label{prop:2vertexinvertible}
Let $S$ be an indecomposable algebra in $\C$ and let $E$ be a faithful, invertible $S$-bimodule $\C$.  Then the tensor algebra $T_{S\oplus S}(E)$  is equivalent to $T_{S\oplus S}(S)$.
\end{proposition}

\begin{proof}
Considering~\eqref{eq:inv} with $X_1 = E$ and $X_2 =S$, we get that 
$$ \bar{X_1}  \otimes_S E_{12}\otimes_S X_2 \;=\;  \bar{E}  \otimes_S E \otimes_S S  \;\cong\;S.$$
This yields the desired result.
\end{proof}

\brk The next consequence of Theorem \ref{thm:param} holds by Theorem~\ref{thm:2categ}. (The minimality condition can be removed via Proposition~\ref{prop:decomp}.)

\begin{corollary}\label{cor:morita}
A categorical Morita equivalence between finite multi-tensor categories $\C$ and $\D$ induces a bijection between equivalence classes of (minimal) $\C$-tensor algebras and of  (minimal) $\D$-tensor algebras.  \qed
\end{corollary}

By Remark~\ref{rmk:RepGam} below, we have such a bijection between (minimal) $\Rep(G)$- and $\mathsf{Vec}_G$-tensor algebras for $G$ a finite group. In general, the data in Theorem~\ref{thm:param} depends intimately on the structure of $\mathcal{C}$; we explore this in Section~\ref{sec:VecGw} for $\mathcal{C} = {\sf Vec}_G^\omega$, and in Sections~\ref{sec:tengt} and~\ref{sec:Path} for $\mathcal{C}$ a group-theoretical fusion category.

\smallskip

This brings us to the following finiteness result in the case when $\mathcal{C}$ is a multi-fusion category, beginning with a remark for $\C$ as above.

\begin{remark}\label{rem:S-faithful}
Recall from Corollary \ref{cor:Adecomp} that an exact algebra $S$ in $\C$ has a unique decomposition  into indecomposable algebras $S=\bigoplus_i S_i$.  For any f.g. $\C$-tensor algebra $T_S(E)$, this gives a decomposition of $S$-bimodules $E= \oplus_{i,j} E_{ij}$ where each $E_{ij} \in {\sf Bimod}_\C(S_i,S_j)$.  Therefore, if $T_S(E)$ is a minimal $\C$-tensor algebra, at most two of the indecomposable summands of $S$ can act nontrivially on $E$.
\end{remark}

\begin{corollary}\label{cor:finitelymany}
In a multi-fusion category $\C$, there are only finitely many minimal, faithful $\mathcal{C}$-tensor algebras, up to equivalence.
\end{corollary}

\begin{proof}
Let $T_S(E)$ be a minimal, $S$-faithful $\mathcal{C}$-tensor algebra.
By Corollary \ref{cor:Adecomp} and Remark \ref{rem:S-faithful}, $S$-faithfulness implies that $S$ has a unique decomposition into a direct sum of either 1 or 2 indecomposable algebras. By Ocneanu rigidity \cite[Corollary~9.1.6(ii)]{EGNO}, the number of choices of these summands is finite, up to equivalence.  Then for any such base algebra $S$, Ocneanu rigidity again applies to the category $\Bimod_\C(S)$ by Proposition~\ref{prop:bimod}.  Since $T_S(E)$ is assumed minimal, there are only finitely many choices of the indecomposable $S$-bimodule $E$ as well, thus finitely many equivalence classes.
\end{proof}

\begin{example}\label{ex:quiverminimal}
Take $\mathcal{C} = {\sf Vec}$. Then the base algebra $S$ of a minimal, faithful $\mathcal{C}$-tensor algebra is Morita equivalent to either $\kk$ or $\kk \times \kk$, by using both the fact that exact $\kk$-algebras are semisimple \cite[Example~7.5.4]{EGNO} and the Artin-Wedderburn theorem.
\begin{itemize}
\item If $S = \kk$, then an indecomposable faithful $S$-bimodule $E$ must be $\kk$. Here, a minimal, faithful $\mathcal{C}$-tensor algebra is unique up to equivalence, and is a path algebra of a loop (isomorphic to $\kk[x]$). 
\smallskip
\item If $S = \kk \times \kk$, then it has two isomorphism classes of indecomposable, faithful bimodules, yielding $\C$-tensor algebras isomorphic to ${\scriptsize \begin{pmatrix} \kk & \kk \\ 0 & \kk \end{pmatrix}}$ and ${\scriptsize \begin{pmatrix} \kk & 0 \\ \kk & \kk \end{pmatrix}}$. But the two bimodules are conjugate by the autoequivalence of $\Mod_\C$-$S$ switching the actions of the two copies of $\kk$.   Thus we have a unique equivalence class of minimal, faithful $\C$-tensor algebra over this $S$; it can be represented by the path algebra of Dynkin type $A_2$.
\end{itemize} 
This shows that each equivalence class of a finitely generated {\sf Vec}-tensor algebra is represented by a path algebra of a finite quiver, as we know from Theorem \ref{thm:eqKQ}.
\end{example}

Finally, we consider a special case of minimal $\C$-tensor algebras $T_S(E)$ where the module category $\M$ of Theorem \ref{thm:param}  (corresponding to $S$) is  $\C$ itself.

\begin{example}\label{ex:extremes}
Fix $\mathcal{C}$ a finite multi-tensor category. Here, we consider equivalence classes of minimal $\mathcal{C}$-tensor algebras $T_S(E)$ for $S = {\bf 1}$, so $\M = {\sf Mod}_\C$-$S$ is $\C$.  (If $\C$ is tensor, we could  take $S=X \otimes X^*$ for any nonzero $X \in \C$; see Example~\ref{ex:S=1}.) 
The tensor equivalence ${\sf Fun}_{\mathcal C}(\mathcal{C}, \mathcal{C}) \sim \mathcal{C}^{\rm op}$ \cite[Example~7.12.3]{EGNO} shows that the equivalence classes of minimal $\mathcal{C}$-tensor algebras of the form $T_{\bf 1}(E)$ are in bijection with conjugacy classes of indecomposable objects in $\C$.
\end{example}


\subsection{Filtered actions of semisimple Hopf algebras on tensor algebras}\label{sec:CHopf}
In this part, we restrict our attention to the case when $\mathcal{C}=\Rep(H)$, the category of finite-dimensional representations of a semisimple Hopf algebra $H$.  We first study Hopf actions of $H$ (see \cite{Montgomery,Radford}) on the $\mathsf{Vec}$-tensor algebras $T_B(V)$ discussed at beginning of Section~\ref{sec:main}; here, $B$ is a finite-dimensional semisimple $\kk$-algebra and $V$ is a $B$-bimodule. We then examine $H$-actions on the degree-completed tensor algebras $\widehat{T}_B(V)$ described in Definition~\ref{def:complete}. We do not assume that these $H$-actions preserve grading (i.e.,  tensor algebras below are not necessarily ${\sf Rep}(H)$-tensor algebras as in Definition~\ref{S,E}).

\smallskip

Recall that $T_B(V)$ is naturally equipped with an ascending filtration as a $\kk$-vector space.  Our first result shows that an $H$-action preserving this filtration is isomorphic to a graded action.

\begin{proposition} \label{prop:filtration}
Let $T_B(V)$ be a $\mathsf{Vec}$-tensor algebra with an action of a semisimple Hopf algebra $H$ that preserves the ascending filtration of $T_B(V)$.  Then $T_B(V)$ is isomorphic as a $\Rep(H)$-module algebra to the $\Rep(H)$-tensor algebra $T_B(W)$, where $W:=(B\oplus V)/B$ as a $B$-bimodule in $\Rep(H)$.  Therefore, the data of Theorem~\ref{thm:param} with $\C=\Rep(H)$ classifies such actions.
\end{proposition}

\begin{proof}
The assumption that the $H$-action preserves the ascending filtration of $T_B(V)$ gives us that $B$ is an algebra in $\Rep(H)$ and $W:=(B\oplus V)/B$ is a $B$-bimodule in $\Rep(H)$.  
So, we get a short exact sequence of $B$-bimodules in $\Rep(H)$, $$0 \to B \to B\oplus V \to W \to 0.$$ 
Since $H$ and $B$ are semisimple, this sequence splits by \cite[Theorem~6]{CF86}. \red{Thus, we get an isomorphism $B\oplus W \cong B\oplus V$ of $B$-bimodules in $\mathsf{Rep}(H)$, as required.} 
\end{proof}

Next, we consider $H$-actions on completed tensor algebras.

\begin{definition} \label{def:complete}
Let $R=T_B(V)$ be a $\sVec$-tensor algebra as above, and $J$ the 2-sided ideal of $R$ generated by $V$.  The {\it degree-completion} $\widehat{T}_B(V)$ is the inverse limit of the system of $\kk$-algebras
\[
\cdots \to R/ J^4 \to R/ J^3 \to R/J^2  \to R/J \cong B.
\]
We denote by $\hat J$ the closure of image of $J$ under the natural injective map $T_B(V) \to \widehat{T}_B(V)$.
\end{definition}

We include a proof of the following lemma, although it is presumably well known.

\begin{lemma}
The Jacobson radical of $\widehat{T}_B(V)$ is the ideal $\hat J$.
\end{lemma}
\begin{proof}
We write $T:=\widehat{T}_B(V)$.  Since $T/\hat J\cong B$ is semisimple, we know $\hat J$ contains the Jacobson radical of $T$.  On the other hand, let $\mathfrak{m}$ be a maximal left ideal in $T$. Then $\mathfrak{m} \supseteq \hat J$, because otherwise we would have $\mathfrak{m}+\hat J=T$, and could write $1=m+j$ for some $m\in \mathfrak{m}$ and $j\in \hat J$. But then  $m=1-j$ would be invertible with inverse $(1-j)^{-1}=1+j+j^2+\cdots$, which makes sense because of the completion.  This would be a contradiction.  So every maximal left ideal contains $\hat J$, thus the Jacobson radical of $T$ does as well.
\end{proof}

The following proposition shows that in the completed setting, $H$-actions automatically preserve the descending filtration by degree.

\begin{proposition}
Any action of $H$ on $\widehat{T}_B(V)$ preserves $\hat J$.
\end{proposition}
\begin{proof} 
We again write $T:=\widehat{T}_B(V)$.  Consider the coaction map $\rho\colon T\to H^*\otimes T$ and consider the composition $\rho'\colon T\to H^*\otimes T/{\hat J}$ of $\rho$ with projection to $T/{\hat J}$. Let $K=\ker \rho'$, which is an ideal of finite codimension in $T$ because the codomain is finite dimensional. Applying the counit of $H^*$, we get the projection $T\to T/{\hat J}$ by the counit axiom, so $K$ is contained in ${\hat J}$. We claim that $K$ is invariant under the $H$-action. Indeed, for $a\in K$ we have by coassociativity $(1\otimes \rho')\rho(a)=(\Delta\otimes 1)\rho'(a)=0$, so $\rho(a)\in \ker(1 \otimes \rho')=H^*\otimes K$. Thus, the action of $H$ on $T$ descends to an action on the finite dimensional algebra $T/K$. Now by \cite[Theorem~2.1]{Linchenko}, the given $H$-action must preserve ${\hat J}/K$, since ${\hat J}/K$ is the Jacobson radical of $T/K$. Thus, $H$ preserves ${\hat J}$.
\end{proof}

\begin{theorem}\label{thm:complete}
Any action of $H$ on $\widehat{T}_B(V)$ is isomorphic to a graded action of $H$ on $\widehat{T}_B(V)$.
\end{theorem}

\begin{proof}
First, by applying the integral of $H$ (i.e., the idempotent of the trivial representation of $H$), 
the natural projection $\hat J \to J/J^2$ splits as a map of $H$-modules.  
This is also a morphism of $B$-bimodules, so the universal property of the tensor algebra gives a map of algebras $T_B(J/J^2) \to \widehat{T}_B(V)$, which is also a map of $H$-modules.  Now the codomain of this morphism is complete, so the universal property of completion induces a map of $H$-module algebras $\widehat{T}_B(J/J^2) \to \widehat{T}_B(V)$, which is an isomorphism.  Since $J/J^2 \cong V$ as $B$-bimodules, this completes the proof.
 \end{proof}
 
\begin{corollary} \label{cor:ssT}
For a tensor algebra $T_B(V)$ in $\mathsf{Vec}$ (i.e., finite-dimensional over $\kk$) and a semisimple Hopf algebra $H$, the data of Theorem~\ref{thm:param} with $\C=\Rep(H)$ classifies all actions of $H$ on $T_B(V)$.
\end{corollary}

\begin{proof}
When $T_B(V)$ is finite-dimensional over $\kk$, it is equal to the degree-completed tensor algebra and the theorem above applies.
\end{proof}

\begin{remark}
Every ${\hat J}$-preserving $H$-action on $T_B(V)$ extends to $\widehat{T}_B(V)$ by continuity, so the data of Theorem~\ref{thm:param} with $\C=\Rep(H)$ can be used to gain some information about $H$-actions of $T_B(V)$ in general.  However, it may be that two nonisomorphic \red{actions  become} isomorphic, or equivalent in the sense of Definition~\ref{def:equiv}, upon extension to the completion (although we do not know an example). So, we do not immediately get a classification.
\end{remark}

\begin{remark}
All results of this section also hold when $\kk$ has positive characteristic, if we impose the additional assumption that $H$ is cosemisimple.
\end{remark}


\section{Tensor algebras in pointed fusion categories} \label{sec:VecGw}

The goal of this section is to study  minimal $\mathcal{C}$-tensor algebras [Definition~\ref{S,E},~\ref{def:minl}] for the pointed fusion categories $\mathcal{C} = {\sf Vec}_G^\omega$.  Due to Theorem~\ref{thm:param}, we first recall results on the classification of indecomposable semisimple ${\sf Vec}_G^\omega$-module categories and on the categories of ${\sf Vec}_G^\omega$-module functors between such module categories (bimodule categories); this is done in Section~\ref{sec:OstNat}. We use this in Section~\ref{sec:VecGw-tensor} to study ${\sf Vec}_G^\omega$-tensor algebras and provide several detailed examples of Theorem~\ref{thm:param} there. 

\brk To classify minimal, faithful $\C$-tensor algebras for a given $\C$, Remark~\ref{rem:S-faithful} allows us to restrict our attention to base algebras with either one or two indecomposable summands. We proceed as such here.


\subsection{Module categories and bimodule categories over ${\sf Vec}_G^\omega$} \label{sec:OstNat} We begin by considering  pointed fusion categories, each of which is equivalent to some category ${\sf Vec}_G^\omega$.
Here, $G$ is a finite group with 3-cocycle \linebreak $\omega\colon G \times G \times G \to \kk^\times$, and objects of {\sf Vec}$_G^\omega$  consist of finite-dimensional $G$-graded $\kk$-vector spaces with associativity constraint determined by $\omega$.
Without loss of generality we assume all $\kk^\times$-cochains are normalized, meaning that they take the value 1 when any coordinate of the input is the group identity.

\begin{notation}[${}^g x$, ${}^g X$, $\psi^g$, $\eta|_L$]\label{not:S4}
We collect a list of the most frequently used notation below.
\begin{itemize}
\item We write ${}^g x := gxg^{-1}$ and ${}^g X:=\{{}^gx \colon x \in X\}$, for an element $g \in G$.
\item Take a 2-cochain $\psi$ on $L \leq G$ and an element $g \in G$. The 2-cochain $\psi^g$ on $L$ is defined by $\psi^g(g_1,g_2) = \psi({}^g g_1, {}^g g_2)$ for $g_1,g_2 \in L$.
\item Let $\eta$ be an $n$-cochain on $G$ and $L \leq G$.  We write $\eta|_{L}$ for the restriction of $\eta$ to $L^{\times n}$ for $n \in \mathbb{N}$.
\end{itemize}
\end{notation}

The following module categories over ${\sf Vec}_G^\omega$ play a central role throughout this work.
Let $G$ be a finite group and $\omega$ a 3-cocycle on $G$.  To a pair $(L,\psi)$ where $L\leq G$ and $\psi$ a 2-cochain on $L$ satisfying $d\psi = \omega|_L$, we assign the indecomposable semisimple ${\sf Vec}_G^\omega$-module category $\mathcal{M}(L, \psi)$ as follows. 
First consider the collection of non-isomorphic simple objects  $\{\delta_{g}\}_{g \in G}$ of ${\sf Vec}_G^\omega$, where $\delta_{g}$ is the 1-dimensional $G$-graded vector space concentrated in degree~$g$.

\begin{definition} [$\mathcal{M}(L, \psi)$, $A(L, \psi)$]
\cite[Example~2.1]{OstrikIMRN2003}  \cite[Example~9.7.2]{EGNO} \label{def:M(La,al)} 
The {\it twisted group algebra} $A(L, \psi)$ in ${\sf Vec}_G^\omega$ is $\bigoplus_{g \in L} \delta_{g}$ as an object in ${\sf Vec}_G^\omega$, with multiplication $\delta_{g} \otimes \delta_{g'} = \psi(g,g') \delta_{gg'}$. We define the ${\sf Vec}_G^\omega$-module category $\mathcal{M}(L, \psi) :={\sf Mod}_{{\sf Vec}_G^\omega}\text{-}A(L, \psi)$.
\end{definition}

If $\omega$ is trivial, then $A(L, \psi)$ is an associative algebra in ${\sf Vec}$. But in general  $A(L, \psi)$ is  an associative algebra only in $\sVec_G^\omega$. 
The following fundamental results of Ostrik and Natale tell us that these are the building blocks of all module categories over ${\sf Vec}_G^\omega$ (up to equivalence), and they give us a criterion for checking when two such module categories are equivalent.

\begin{proposition} \label{prop:OstNat} \cite[Example~2.1]{OstrikIMRN2003} \cite{Natale2016} 
Every indecomposable semisimple module category over the fusion category ${\sf Vec}_G^\omega$ is equivalent to some $\M(L,\psi)$, where
\begin{itemize}
\item $L$ is a subgroup of $G$ such that the class of $\omega|_{L}$ is trivial in $H^3(L, \kk^\times)$, and 
\item $\psi\colon L \times L \to \kk^\times$ is a 2-cochain on $L$ satisfying $d\psi = \omega|_{L}$. 
\end{itemize} 
Moreover, $\M(L, \psi)$ and $\M(L', \psi')$ are equivalent as ${\sf Vec}_G^\omega$-module categories if and only if there exists $g \in G$ such that
$L = {}^g L'$ and the class of the 2-cocycle $\psi'^{-1}\psi^g \Omega_g|_{L'}$ is trivial in $H^2(L', \kk^\times)$, where $\psi^g$ is as in Notation \ref{not:S4}.
Here, 
$\Omega_g\colon G \times G \to \kk^\times$ is given by
\begin{equation}\label{eq:Omegag}
\Omega_g(g_1,g_2) = \frac{\omega({}^g g_1, {}^g g_2, g) \; \omega(g, g_1, g_2)}{\omega({}^g g_1, g, g_2)}.
\end{equation}
\vspace{-.3in}

\qed
\end{proposition}

\begin{remark} \label{rmk:RepGam}
Indecomposable semisimple module categories over {\sf Rep}($G$) are determined by the same data of Proposition~\ref{prop:OstNat} for ${\sf Vec}_G$.
Namely, {\sf Rep}($G$) and ${\sf Vec}_G$ are categorically Morita equivalent since \linebreak $({\sf Vec}_G)^*_{{\sf Vec}} \sim {\sf Rep}(G)$ \cite[Example~7.12.19]{EGNO}, and as a result, there is a 2-equivalence between the 2-category of semisimple module categories over ${\sf Vec}_G$ and over ${\sf Rep}(G)$ [Theorem~\ref{thm:2categ}]. See \cite[Proposition~2.3]{OstrikIMRN2003}, a result due to M\"{u}ger~\cite{Muger}, and also \cite[Example~7.4.9, Corollary~7.12.20]{EGNO} for direct descriptions of indecomposable semisimple module categories over $\Rep(G)$.
\end{remark}

We summarize some elementary observations in the following lemma which are useful for applying the proposition above in specific examples.
\begin{lemma} \label{lem:Omega}
Suppose $L \leq G$ is a subgroup such that $\omega|_{L}$ is trivial in $H^3(L, \kk^\times)$.  Then the following hold.
\begin{enumerate}
\item[\textnormal{(a)}] For every conjugate ${}^g L$ of $L$ in $G$, we have that $\omega|_{{}^gL}$ is trivial in $H^3({}^gL, \kk^\times)$.
\smallskip

\item[\textnormal{(b)}] If $H^2(L', \kk^\times)$ is trivial, then the pairs $(L', \psi')$ yielding module categories equivalent to that of $(L, \psi)$ are exactly those with $L = {}^g L' $ for some $g \in G$.
\end{enumerate}
\end{lemma}

\begin{proof}
(a) Let $\psi \in H^2(L, \kk^\times)$ be such that $d\psi = \omega|_L$.  It can be directly checked that $d(\psi^{g^{-1}} \Omega_{g^{-1}}) = \omega|_{{}^g L}$.

\smallskip

(b) It is immediate since the condition comparing $\psi, \psi'$ in Proposition \ref{prop:OstNat} is vacuous in this case.
\end{proof}

\begin{remark}
We also draw the reader's attention to Natale's example \cite[Example 3.6]{Natale2016}, which shows that the 2-cocycle $\Omega_g$ must be considered when computing equivalence classes of indecomposable semisimple module categories: Even when $\psi^{-1} \psi'$ is not a coboundary, it is possible for $(L,\psi)$ and $(L,\psi')$ to yield equivalent module categories. In order to get $(L, \psi) \not \sim (L, \psi')$, one must check for each $g \in G$ with ${}^g L = L$ that $\psi'^{-1}\psi^g \Omega_g|_{L}$ is nontrivial in $H^2(L, \kk^\times)$.
\end{remark}

Now to describe the minimal ${\sf Vec}_G^\omega$-tensor algebras  $T_S(E)$ up to equivalence, we may take $S$ to be a direct sum of twisted group algebras $A(L, \psi)$ in ${\sf Vec}_G^\omega$.

\begin{example} \label{extreme-Gam}
Continuing Example~\ref{ex:extremes}, note that if the pair $(L, \psi)$ corresponds to the ${\sf Vec}_G^\omega$-module category ${\sf Vec}_G^\omega$ itself, then $(L, \psi) = (\langle e \rangle, 1)$. 
\end{example}

Next, we consider the generating bimodules of minimal $\mathsf{Vec}_G^\omega$-tensor algebras by recalling a result of Ostrik that classifies categories of functors between indecomposable semisimple $\mathsf{Vec}_G^\omega$-module categories. By Remark \ref{rem:S-faithful}, it is enough to consider $(S_i, S_j)$-bimodules, where $S_i:=A(L_i, \psi_i)$ and $S_j:=A(L_j, \psi_j)$ are indecomposable semisimple algebras. Also for the result below, we follow \cite[Section~3.1]{Karp2} and call a projective representation of a group $G$ with Schur multiplier $\psi$ a {\it $\psi$-representation of $G$}.

\begin{notation}[${L}^g_{i,j}$]\label{not:intersection}
Given two subgroups $L_i, L_j \leq G$ and $g \in G$, we write $L^g_{i,j}:= L_i \cap g L_j g^{-1}$.
\end{notation}

\begin{proposition}[$\psi^g_{i,j}=:(\psi_i\psi_j^{-1})_g$, $m_{i,j}(g)$, $M(g, \rho)$] \label{prop:Fun-Vec} \cite[Proposition~3.2]{OstrikIMRN2003} \cite[Theorem~5.1]{GN}
\linebreak Let $(L_i, \psi_i)$, $(L_j, \psi_j)$ be two pairs  determining indecomposable semisimple $\mathsf{Vec}_G^\omega$-module categories as in Proposition~\ref{prop:OstNat}.
For each $g \in G$, the group $L^g_{i,j}$ has a well-defined 2-cocycle
\begin{equation}\label{eq:aijg}
\psi^g_{i,j}(\ell,\ell'):=(\psi_i\psi_j^{-1})_g(\ell,\ell'):=\, 
\psi_i(\ell,\ell')\,\cdot\,\psi_j(g^{-1}\ell'^{-1}g,\, g^{-1}\ell^{-1}g)\,\cdot
\frac{\omega(\ell, \ell', g)\, \cdot\, \omega(\ell,\, \ell'g,\, g^{-1} \ell'^{-1}g)}
{\omega(\ell\ell'g,\, g^{-1}\ell'^{-1}g,\, g^{-1}\ell^{-1}g)},
\end{equation}
for $\ell, \ell' \in L_{i,j}^g$. 

\brk
Moreover, let $\{g_k\}_{k \in L_i \backslash G / L_j}$ be a set of representatives of $L_i$-$ L_j$ double cosets in $G$. Then the rank of the category ${\sf Fun}_{\mathsf{Vec}_G^\omega}(\M(L_i, \psi_i), \M(L_j, \psi_j))$, or equivalently of the category of  $(A(L_i,{\psi_i}), A(L_j,{\psi_j}))$-bimodules in $\mathsf{Vec}_G^\omega$, is equal to 
\begin{equation}\label{eq:arrow}\textstyle \sum_{k \in L_i \backslash G / L_j}\; m_{i,j}(g_k)
\end{equation}
where $m_{i,j}(g)$ is the \red{number} of linear equivalence classes of irreducible $\psi^g_{i,j}$-representations of the group $L^g_{i,j}$.

\brk
We get that the simple objects of this category are classified by the pairs $(g_k, \rho)$, where $\rho$ is an irreducible $\psi_{i,j}^{g_k}$-representation of the group $L_{i,j}^{g_k}$.  The corresponding simple object $M(g_k,\rho) \in \mathsf{Vec}_G^\omega$ is supported \red{on the double coset} $L_i g_k L_j$.
\qed
 \end{proposition}

We introduced two notations above for the same 2-cocycle because the first is more convenient in this section, while the latter notation is more convenient in Sections~\ref{sec:tengt} and~\ref{sec:Path}.

\brk
Now to understand the rank count above,
recall that for a 2-cocycle $\phi$ on a group $G$, an element $g \in G$ is called $\phi$-regular when $\phi(g,h) = \phi(h,g)$ for all $h \in C_G(g)$, the centralizer of $g$ in $G$ \cite[Section~2.6]{Karp2}.
Furthermore, $\phi$-regularity of an element $g$ depends only on the cohomology class of $\phi$, and if $g$ is $\phi$-regular, then so is every conjugate of $g$ \cite[Lemma~2.6.1]{Karp2}.
The following result of Schur can be found in \cite[Theorem~6.1.1]{Karp2}.

\begin{theorem}[Schur] \label{thm:schur}
Let $L$ be a group and $\psi \in Z^2(L, \kk^\times)$ a 2-cocycle.
Then the number of linear equivalence classes of irreducible $\psi$-representations of $L$
is equal to the \red{number} of $\psi$-regular conjugacy classes of~$L$.
 \qed
 \end{theorem}

We also introduce some terminology that will be used.

\begin{definition}\label{def:Schurtriv}
A finite group $G$ is called {\it Schur-trivial} if $H^2(G,\kk^\times) = 1$, and we call $G$ {\it sub-Schur-trivial} if each of its subgroups is Schur-trivial.
\end{definition}


\subsection{${\sf Vec}_G^\omega$-tensor algebras} \label{sec:VecGw-tensor}

Now we study minimal, $S$-faithful ${\sf Vec}_G^\omega$-tensor algebras $T_S(E)$.  By Remark~\ref{rem:S-faithful}, we know that $S$ has at most two summands, and is therefore Morita equivalent to $A(L, \psi)$ or  \linebreak $A(L_i, \psi_i) \oplus   A(L_j, \psi_j)$.

\subsubsection{Trivial $\omega$.} We first consider the case when $\omega$ is cohomologically trivial.  Without loss of generality we assume throughout this section that $\omega=1$, the constant cochain, and note that ${\sf Vec}_G^1={\sf Vec}_G$.  The classifying data for tensor algebras here is the same as in $\Rep(G)$; see Remark~\ref{rmk:RepGam}.

\begin{proposition}
\label{prop:omegatrivial}
Let $[G]$ be a set of representatives of the conjugacy classes of subgroups $L$ of $G$, let $N_G(L)$ be the normalizer of $L$ in $G$, and let $|X/\Gamma|$ be the number of orbits of a group $\Gamma$ acting on a set $X$.
Then the number of indecomposable semisimple base algebras in ${\sf Vec}_G$, up to equivalence, is
\begin{equation}\label{eq:omegatrivialcount}
\sum_{L \in [G]} |H^2(L, \kk^\times) / N_G(L)|.
\end{equation}
\end{proposition}
\noindent So, if $G$ is sub-Schur-trivial, then \eqref{eq:omegatrivialcount} is equal to $\# [G]$; if, further, $G$ is abelian, then  \eqref{eq:omegatrivialcount} equals $\# (L \leq G)$.
\begin{proof}
We consider Proposition~\ref{prop:OstNat} in the case $\omega=1$, noting that $\omega|_L =1$ on any $L\leq G$.
Therefore, by Proposition~\ref{prop:OstNat} the collection of indecomposable module categories over the fusion category ${\sf Vec}_G$ is parametrized by  conjugacy classes of pairs $(L, \psi)$ where $L \leq G$ is any subgroup and $\psi \in H^2(L, \kk^\times)$.
The relation $(L, \psi) \sim (L', \psi')$ simplifies to just 
$L = {}^g L'$ and $\psi'=\psi^g \in H^2(L', \kk^\times)$ for some $g \in G$, since each $\Omega_g=1$.
Then the count \eqref{eq:omegatrivialcount} follows from the orbit-stabilizer theorem applied to the conjugation action of $G$ on the set of pairs $(L, \psi)$ as above.
\end{proof}

Next, we note a simplification for counting bimodules which occurs in the case when $G$ is abelian.

\begin{lemma}\label{lem:omegatrivialbimod}
Suppose $G$ is abelian and let $(L_i, \psi_i), (L_j,\psi_j)$ be as in Proposition~\ref{prop:Fun-Vec}.  Then the number of indecomposable $(A(L_i,{\psi_i}), A(L_j,{\psi_j}))$-bimodules in $\mathsf{Vec}_G$,
 up to isomorphism, is
\begin{equation}\label{eq:omegatrivialbimodcount1}
\frac{|G||L_i \cap L_j|}{|L_i| |L_j|}(\text{\red{number} of }\psi^e_{i,j}\text{-regular conjugacy classes in }L_i \cap L_j).
\end{equation}
 \end{lemma}
\begin{proof}
Since $G$ is abelian, the number of representatives of $(L_i, L_j)$-double cosets in $G$  in Proposition~\ref{prop:Fun-Vec} is $(|G||L_i \cap L_j|)/(|L_i| |L_j|)$.  Since $\omega=1$ and conjugation in $G$ is trivial, $\psi^g_{i,j} = \psi^e_{i,j}$ for all $g \in G$.  Then the count follows from \eqref{eq:arrow} and Theorem \ref{thm:schur}.
\end{proof}

By the two previous results, the examination of ${\sf Vec}_G$ is easier in the cases when $G$ is abelian or is sub-Schur-trivial. 
 So we consider these cases in the four examples below, where $G$ is:
\begin{itemize}
\item a cyclic group $\mathbb{Z}_n$ (that is, abelian and sub-Schur-trivial); 
\item the Klein-four group $\mathbb{Z}_2 \times \mathbb{Z}_2$ (that is, abelian and not Schur-trivial); 
\item the symmetric group $S_3$ of order 6 (that is, non-abelian and sub-Schur-trivial); and 
\item the dihedral group $D_8$ of order 8 (that is, neither abelian nor Schur-trivial).
\end{itemize}

\begin{example} \label{ex:Zn}
Take $G = \langle x ~|~ x^n =e \rangle \cong \mathbb{Z}_n$, which is abelian and sub-Schur trivial (see, e.g., \cite[Proposition~10.1.1(ii)]{Karp2}).   
Let $\tau(n)$ denote the set of positive integers dividing $n$.
If $T_S(E)$ is a minimal, $S$-faithful tensor algebra in $\C$, then 
we may take $S = A(\langle x^m\rangle, 1)$ or $S=A(\langle x^m\rangle, 1) \oplus A(\langle x^{m'}\rangle, 1)$ for $m, m' \in \tau(n)$ by Proposition~\ref{prop:omegatrivial}.

Suppose $S = A(\langle x^m\rangle, 1)$ for $m \in \tau(n)$. 
From Proposition~\ref{prop:Fun-Vec}, the isomorphism classes of $S$-bimodules $E$ are classified by pairs $(g_k,\rho)$ where $g_k$ is a representative of a $\langle x^m\rangle$-$\langle x^m\rangle$-double cosets (i.e. a $\langle x^m\rangle$-coset) in $G$, and $\rho$ an irreducible representation of $\langle x^m\rangle$.  Since $G$ is abelian,  all its irreducible representations are 1-dimensional, and it is easy to see that nonisomorphic bimodules will not be conjugates of one another. So, equivalence classes of tensor algebras with base algebra $S$ are in bijection with isomorphism classes of $S$-bimodules in this case.

On the other hand, suppose $S = S_1\oplus S_2$ with $S_1=A(\langle x^m\rangle, 1)$ and $S_2= A(\langle x^{m'}\rangle, 1)$ for $m,m' \in \tau(n)$.  Then given any pair $E, E' \in \Bimod_\C(S_1, S_2)$, it can be shown by direct computation that there exist invertible $S_i$-bimodules $X_i$ such that $E' \cong \bar{X_1}  \otimes_S E \otimes_S X_2$ (where $\bar{X_1}$ is the inverse of $X_1$ in $\Bimod_\C(S_1)$ as before).  Thus for each $S$ with two indecomposable summands, there is only one equivalence class of minimal, faithful tensor algebra $T_S(E)$.
\end{example}

\medskip 

Next, we consider an abelian group which has a cohomologically nontrivial 2-cocycle.

\begin{example} \label{ex:Z2xZ2} 
Take $G = \langle x,y : x^2 = y^2 = e, ~yx=xy \rangle \cong \mathbb{Z}_2 \times \mathbb{Z}_2 $. There are five subgroups $L$ of $G$ up to conjugacy:
$\{\langle e \rangle, ~\langle x \rangle, ~\langle y \rangle, ~\langle xy \rangle, ~G\}$, and $H^2(L, \kk^\times)$ is trivial for $L \neq G$.  However, 
$H^2(G, \kk^\times) \cong \mathbb{Z}_2$ (see, e.g., \cite[Proposition~10.7.1]{Karp2}), and a 2-cochain on $G$ representing the nontrivial element of $H^2(G, \kk^\times)$ is \begin{equation} 
\label{eq:Z2xZ2}
\mu(x^{i_1}y^{j_1}, x^{i_2}y^{j_2}) = (-1)^{j_1i_2}, \quad \quad 0 \leq i_\ell, j_\ell \leq 1.
\end{equation}
By Proposition \ref{prop:omegatrivial}, there are 6 indecomposable semisimple base algebras, up to equivalence.

Turning to bimodules, we can compute the quantity \eqref{eq:omegatrivialbimodcount1} (or equivalently, \eqref{eq:arrow}) for every pair of subgroups of $G$, and the result is 
summarized in the following table.  
\smallskip

{\footnotesize \[
\begin{array}{c|c|c|c|c|c|c|}
\begin{array}{c}
\text{The value}~\eqref{eq:arrow} \text{ for}\\(L_i, \psi_i) \downarrow \text{ and }~~(L_j, \psi_j) \rightarrow
\end{array} & (\langle e \rangle,1) & (\langle x \rangle,1) & (\langle y \rangle,1) & (\langle xy \rangle,1) & (G,1) & (G,\mu) \\
\hline
(\langle e \rangle,1) & 4 & 2 & 2 & 2 & 1 & 1\\
\hline
(\langle x \rangle,1) & 2 & 4 & 1 & 1 & 2 & 2\\
\hline
(\langle y \rangle,1) & 2 & 1 & 4 & 1 & 2 & 2\\
\hline
(\langle xy \rangle,1) & 2 & 1 & 1 & 4 & 2 & 2\\
\hline
(G,1) & 1 & 2 & 2 & 2 & 4 & 1\\
\hline
(G,\mu)& 1 & 2 & 2 & 2 & 1 & 4\\
\hline
\end{array}
\]}
\smallskip

\noindent The count is simplified by noting that $\psi_{i,j}^e$ is cohomologically trivial unless 
$L_i = L_j = G$ and either $\psi_i$ or $\psi_j$ is $\mu$.
When $\psi_{i,j}^e$ is cohomologically trivial, the quantity \eqref{eq:omegatrivialbimodcount1} is $4\frac{|L_i \cap L_j|^2}{|L_i| |L_j|}$ and so it can be directly computed case-by-case.
In the three cases where $\psi_{i,j}^e$ is cohomologically nontrivial, we first have that $|G \backslash G / G| = 1$. Moreover, when $\psi_i = \psi_j = \mu$, we get that $\psi_{i,j}^e$ is symmetric so the \red{number} of $\psi_{i,j}^e$-regular conjugacy classes in $G$ is~4. On the other hand, when only one of $\psi_i$ or $\psi_j$ is $\mu$, then it can be directly seen from \eqref{eq:Z2xZ2} that the only $\psi_{i,j}^e$-regular conjugacy class is $\{e\}$.
Now by taking cases where $S$ is either an indecomposable algebra in ${\sf Vec}_{\mathbb{Z}_2 \times \mathbb{Z}_2}$ or is a direct sum of two indecomposable algebras in ${\sf Vec}_{\mathbb{Z}_2 \times \mathbb{Z}_2}$, one can classify equivalence classes of $S$-faithful ${\sf Vec}_{\mathbb{Z}_2 \times \mathbb{Z}_2}$-tensor algebras as done in Example~\ref{ex:Zn}.
\end{example}

Now we consider a nonabelian group where every subgroup has trivial second cohomology.

\begin{example} \label{ex:S3}
Take $G = \langle r, s ~|~ r^3 = s^2 = (sr)^2 =e\rangle\cong S_3$, the symmetric group on three letters, 
noting that there are four subgroups $L$ of $G$, up to conjugacy:
$\{\langle e \rangle, ~\langle s \rangle,  ~\langle r \rangle,~G \}$. Here, $H^2(L, \kk^\times)$ is trivial for  all $L \leq G$ (see, e.g. \cite[Proposition~10.1.1(ii) or Theorem~12.2.2]{Karp2}),
so Proposition \ref{prop:omegatrivial} implies that there are four indecomposable semisimple base algebras $S$ in ${\sf Vec}_{G}$, up to equivalence.
To illustrate equivalence versus isomorphism of tensor algebras, we will describe the equivalence classes of $\C$-tensor algebras of the form $T_S(E)$ and $T_{S\oplus S}(E)$ for each indecomposable $S$.  Similar methods can be applied to the cases where $S=S_1 \oplus S_2$ with $S_1$ and $S_2$ nonequivalent.  The descriptions of the bimodules below all follow from Proposition \ref{prop:Fun-Vec}.

For $S=A(\langle e \rangle)$, the 6 indecomposable bimodules are $\{\delta_g : g \in S_3\}$ as objects in $\C$, and are all invertible.  Conjugation of bimodules is conjugation in $S_3$, so there are only 3 equivalence classes of tensor algebras of the form $T_S(E)$.  On the other hand, there is just one equivalence class of minimal, faithful tensor algebra $T_{S\oplus S}(E)$  by Proposition \ref{prop:2vertexinvertible}.

For $S=A(\langle s \rangle)$, there are 3 isoclasses of indecomposable bimodules.  Two of them are $\delta_e \oplus \delta_s$ as objects in $\C$, one associated to the trivial representation and the other associated to the sign representation of $\langle s \rangle$.  The other indecomposable bimodule is $\delta_r \oplus \delta_{r^2} \oplus \delta_{sr} \oplus \delta_{sr^2}$ as an object of $\C$.  The first two are invertible and conjugation by them acts trivially, so there are 3 conjugacy classes of minimal tensor algebras of the form $T_S(E)$.  On the other hand, there are just 2 equivalence classes of minimal tensor algebras $T_{S\oplus S}(E)$ because the first two are equivalent by Proposition \ref{prop:2vertexinvertible}.

For $S=A(\langle r \rangle)$, there are 6 isoclasses of indecomposable bimodules: three of them are $\delta_e \oplus \delta_r \oplus \delta_{r^2}$ as objects in $\C$, associated to the 3 irreducible representations of $\langle r \rangle$, and the other three are $\delta_s \oplus \delta_{sr} \oplus \delta_{sr^2}$ as objects in $\C$, also associated to the 3 irreducible representations of $\langle r \rangle$.  All 6 of these bimodules are invertible. It can be computed that this group of invertible bimodules is isomorphic to $\mathbb{Z}_6$ (generated by $\delta_s \oplus \delta_{sr} \oplus \delta_{sr^2}$ with either nontrivial irreducible representation of $\langle r \rangle$).  Thus conjugation is trivial, giving 6 equivalence classes of minimal tensor algebras $T_S(E)$.  Again by Proposition \ref{prop:2vertexinvertible} there is only one equivalence class of minimal tensor algebras $T_{S\oplus S}(E)$.

For $S=A(S_3)$, there are 3 isoclasses of indecomposable bimodules, all having underlying object $\bigoplus_{g \in S_3} \delta_g$ in $\C$, associated to the 3 irreducible representations of $S_3$.  The two 1-dimensional representations give invertible bimodules and again conjugation by them is trivial, so there are 3 equivalence classes of faithful, minimal tensor algebras $T_S(E)$ and 2 equivalence classes of faithful, minimal tensor algebras $T_{S\oplus S}(E)$.
\end{example}

Finally, we consider a nonabelian group in which some subgroups have non-trivial second cohomology.

\begin{example} \label{ex:D8}
Take $G$ to be the dihedral group $D_8$ of order~8, with presentation
\begin{equation}\label{eq:D8}
D_8 = \langle x, y, z~|~ x^2=y^2=z^2 = e, ~ xy=yx, ~zx=yz, ~zy=xz\rangle.
\end{equation}
Note that there are 8 subgroups $L$ of $G$, up to conjugacy:
$$\{\langle e \rangle, \quad \langle x \rangle={}^z\langle y \rangle,  \quad \langle z \rangle={}^{yz}\langle xyz \rangle, \quad \langle xy \rangle, \quad \langle x,y \rangle, \quad \langle xy, z \rangle,  \quad \langle xz \rangle, \quad G \}.$$
Also, $H^2(G, \kk^\times) \cong \mathbb{Z}_2$ by \cite[Corollary~10.1.27]{Karp2}. The nontrivial cohomology class is represented by
\begin{equation}\label{eq:D82cocycle}
\beta(x^{i_1} y^{j_1} z^{n_1}, x^{i_2} y^{j_2} z^{n_2}) = (-1)^{j_1i_2}, \quad 0 \leq i_\ell,j_\ell,n_\ell \leq 1.
\end{equation}
By Proposition \ref{prop:omegatrivial},  up to equivalence, the indecomposable semisimple base algebras are 
represented by the pairs $(L, \psi)$:

\vspace{-.05in}

{\footnotesize 
\begin{equation*}
\begin{split}
\{(\langle e \rangle, 1), \hspace{.05in} (\langle x \rangle,1),  \hspace{.05in} (\langle z \rangle,1), \hspace{.05in} (\langle xy \rangle,1),
~\hspace{.05in} (\langle x,y \rangle,1),  \hspace{.05in} (\langle x,y \rangle, \beta|_{\langle x,y \rangle}), \hspace{.05in}
(\langle xy,z \rangle,1), \hspace{.05in} (\langle xy,z \rangle, \beta|_{\langle xy,z \rangle}),  \hspace{.05in} (\langle xz \rangle,1), \hspace{.05in} (D_8,1),  \hspace{.05in} (D_8,\beta)\}.
\end{split}
\end{equation*}}

\vspace{-.15in}

We carry out the count of isomorphism classes of indecomposable bimodules for some special cases, but we leave bimodule conjugacy class computations to the reader.

\begin{enumerate} 
\item At one extreme, we can take $(L_i, \psi_i)=(\langle e \rangle, 1)=(L_j,\psi_j)$.  The set of coset representatives $\{g_k\}$ in Proposition \ref{prop:Fun-Vec} is the entire group $D_8$.  Each $L_{i,j}^{g_k}=\langle e \rangle$, so each $m_{i,j}(g_k)=1$ and the total count of indecomposable bimodules is $|D_8|=8$.

\smallskip

\item Now consider $(L_i, \psi_i)=(\langle x,y \rangle, 1)$ and $(L_j,\psi_j)=(\langle z \rangle, 1)$.  In this case we have $\{g_k\}=\{e\}$ and $L_{i,j}^e=\langle e \rangle$ with $m_{i,j}(e)=1$, so there is a unique bimodule for this pair.

\smallskip

\item If we take $(L_i, \psi_i)=(\langle xy,z \rangle, 1)$ and $(L_j,\psi_j)=(\langle z \rangle, 1)$, then $\{g_k\}=\{e,x\}$. We have $L_{i,j}^e = \langle z \rangle$ and $L_{i,j}^x=\langle e \rangle$. So, $m_{i,j}(e)=2$ and $m_{i,j}(x)=1$, and there are three bimodules for this pair.

\smallskip

\item Next, we examine the case $L_i=L_j=\langle x,y\rangle$ for various choices of $\psi_i, \psi_j$.  We have $\{g_k\}=\{e, z\}$ and $L_{i,j}^e=L_{i,j}^z=\langle x, y\rangle$.  Note that $H^2(\langle x,y\rangle,\kk^\times)\cong \mathbb{Z}_2$  is equal to $\{1, \beta|_{\langle x,y \rangle}\}$.
If $\psi_i=\psi_j=1$ or $\psi_i=\psi_j=\beta|_{\langle x,y \rangle}$, then $\psi_{i,j}^e = \psi_{i,j}^z=1$, and hence, $m_{i,j}(e) + m_{i,j}(z)=4+4 = 8$. Else, if either
 $\psi_i=\beta|_{\langle x,y \rangle}$ and $\psi_j=1$, or, $\psi_i=1$ and $\psi_j=\beta|_{\langle x,y \rangle}$, then $\psi_{i,j}^e = \psi_{i,j}^z=\beta|_{\langle x,y \rangle}$; in this case, $m_{i,j}(e) + m_{i,j}(z)=1+1 = 2$.
This uses the same reasoning as in Example \ref{ex:Z2xZ2}.
 
\smallskip

\item Now if $(L_i, \psi_i)=(L_j,\psi_j) = (D_8, 1)$,  then $\{g_k\}=\{e\}$.  By Theorem \ref{thm:schur} we get that $m_{i,j}(e)~=~5$, the \red{number} of conjugacy classes of elements in $D_8$, or equivalently, the \red{number} of irreducible representations of $D_8$.
\end{enumerate}
\end{example}

Example \ref{ex:D8} will be used in Section~\ref{sec:H8} in the study of $\C$-tensor algebras for $\C$ being the category of finite-dimensional representations of the Kac-Patjutkin Hopf algebra $H_8$. 

\subsubsection{Nontrivial  $\omega$.} 

We now consider  $\omega \in Z^3(G, \kk^\times)$ cohomologically nontrivial. Note that the pairs $(L, \psi)$ parametrizing indecomposable semisimple ${\sf Vec}_G^\omega$-module categories in Proposition~\ref{prop:OstNat} are highly dependent on the choice of $\omega$. In particular, there will typically be fewer $L$ for $\omega$ nontrivial as compared to the case $\omega=1$ because of the requirement that $\omega$ restricted to $L$ must be cohomologically trivial. 

\red{We study the case} when $G = \mathbb{Z}_n$ and when $G = D_8$ for a specific $\omega \in  H^3(D_8, \kk^\times)$ used later in Section~\ref{sec:H8}. We leave other examples for the reader.
In the examples below, $$\langle d\rangle_t \text{ denotes the remainder of } d  \text{ modulo } t.$$

\begin{example}
Continuing Example~\ref{ex:Zn},  take $G =  \langle x~|~x^n = e \rangle \cong \mathbb{Z}_n$. Let $\zeta$ be a primitive $n^{\rm th}$ root of 1. By \cite[(2.3.18)]{dWP} or \cite[Example 2.6.4]{EGNO}, the cohomology classes of 3-cocycles on $G$ are represented~by
\begin{equation}\label{eq:omegal}
\omega_{\ell}(x^i,x^{j},x^{k}) = \zeta^{\ell i(j+k - \langle j+k \rangle_n)/n}, \quad \text{ for }\ell = 0, 1, \dots, n-1.
\end{equation}
Since $\omega$ is cohomologically nontrivial in this section, we take $\ell > 0$. 
Let $\tau(n)$ be the set of positive divisors of $n$. Recall that the distinct subgroups of $G$ are $\langle x^{m} \rangle \cong \mathbb{Z}_{n/m}$ for $m \in \tau(n)$.  Fix $L$ such a subgroup.

Let us consider the restriction of $\omega_\ell$ to $L$.
We can write
\begin{equation}\label{eq:omegalm}
\omega_{\ell}(x^{mi},x^{mj},x^{mk}) = (\zeta^{m})^{\ell i(mj+mk - \langle mj+mk \rangle_n)/n}.
\end{equation}
One can check that $j+k-\langle mj+mk\rangle_n/m=j+k-\langle j+k \rangle_{n/m}$, and thus we can rewrite \eqref{eq:omegalm} in the standard form \eqref{eq:omegal} applied to the cyclic group $\langle x^m\rangle$, noting that $\zeta^m$ is a primitive $(n/m)^{\rm th}$ root of 1:
\begin{equation*}
\omega_{\ell}(x^{mi},x^{mj},x^{mk}) = (\zeta^{m})^{\ell i(j+k-\langle j+k \rangle_{n/m})/(n/m)}
\end{equation*}
This shows that $\omega_\ell$ restricted to $L$ is cohomologically trivial if and only if $|L|=n/m$ divides $\ell$.
So by Proposition \ref{prop:OstNat} and Lemma~\ref{lem:Omega}(b), the indecomposable semisimple base algebras $S$ in ${\sf Vec}_{G}^{\omega_\ell}$ are exactly the algebras $A(L, 1)$ where $L =\langle x^{m} \rangle$ and $n/m$ divides $\ell$.

The study of minimal, faithful tensor algebras in ${\sf Vec}_{G}^{\omega_\ell}$ can then be carried out in exactly the same way as in Example \ref{ex:Zn}, but using only base algebras with summands as described in the previous paragraph.
\end{example}

\begin{example} \label{ex:D8nontriv}
Continuing Example~\ref{ex:D8} with $G=D_8$, consider the nontrivial 3-cocycle $\omega \in Z^3(G, \kk^{\times})$ from \cite[Section~3.5]{Kac1968} given below; this 3-cocycle will be of use in the next two sections. For subgroups $$K = \langle z ~|~ z^2=e\rangle \quad \text{and} \quad N = \langle x, y ~|~ x^2=y^2=e, ~xy=yx\rangle$$ of $G$, let $e_{i,j}$ be the dual basis for $\kk^N$ where $e_{i,j} (x^k y^l) = \delta_{i,k} \delta_{j,l}$ and consider the maps 
$$\sigma\colon K \times K \rightarrow (\kk^N)^\times \quad \text{and} \quad \tau\colon N \times N \rightarrow (\kk^K)^\times$$ 
defined as follows.  Put $\sigma_{i,j}(h,h') := \sigma(h,h')(e_{i,j})$ and $\tau_{z^n}(p,p') := \tau(p,p')(z^n)$ for $h,h' \in K$ and $p, p' \in N$. \red{Define  $\sigma$ and $\tau$  by} setting the value equal to 1 except for the following:
\begin{equation}
\sigma_{1,0}(z,z) = \sigma_{0,1}(z,z) =-\sqrt{-1},
\end{equation}
\begin{equation}\label{eq:tau}
\tau_z(x,x) = \tau_z(y,y) = \tau_z(x, xy) = \tau_z(xy,y) = \sqrt{-1}, \quad \tau_z(y,x) = -1, \quad \tau_z(xy, x) = \tau_z(y, xy) = -\sqrt{-1}.
\end{equation}
Now let
\begin{equation} \label{eq:D8nontriv}
\omega(x^{i_1}y^{j_1}z^{n_1},~ x^{i_2}y^{j_2}z^{n_2}, ~x^{i_3}y^{j_3}z^{n_3})
= \sigma_{i_1,j_1}(z^{n_2}, z^{n_3}) \tau_{z^{n_3}}(x^{j_1}y^{i_1},x^{i_2}y^{j_2})
 \quad 0 \leq n_\ell, i_\ell , j_\ell \leq 1.
\end{equation}
Recall the 8 conjugacy classes of subgroups of $G$ listed in Example~\ref{ex:D8}.
We use formula \ref{eq:D8nontriv} to directly compute that $\omega|_{L}$ is trivial in $H^3(L,\kk^\times)$ exactly when $L$ is one of the following: 
\begin{equation}\label{eq:D8subgroups}
\{\langle e \rangle, \quad \langle x \rangle,  \quad \langle xy \rangle, \quad \langle z \rangle,
\quad \langle x,y \rangle, \quad \langle xy,z \rangle\}.
\end{equation}

Since the first four of these subgroups have trivial Schur multiplier, we get four nonequivalent indecomposable semisimple base algebras $S$ from these.  For the remaining two subgroups, their Schur multipliers are each isomorphic to $\mathbb{Z}_2$, so we must consider the possibilities for equivalence  as in Proposition \ref{prop:OstNat}.  It turns out that the two different choices of cocycle end up giving equivalent module categories, just as in \cite[Example~3.6]{Natale2016}.  In more detail,
first consider $L=\langle x,y \rangle$.
From a direct substitution of \eqref{eq:D8nontriv} into \eqref{eq:Omegag} and from the definitions of $\sigma, \tau$, it can be calculated that
\begin{equation*}
\Omega_z(x^{i_1}y^{j_1},\, x^{i_2}y^{j_2}) = \tau_z (x^{i_1}y^{j_1},\, x^{j_2}y^{i_2}), 
\end{equation*}
and that this represents a nontrivial cohomology class on $\langle x, y \rangle$.  Therefore, the two pairs $(\langle x,y \rangle, 1)$ and $(\langle x,y \rangle, \beta|_{\langle x,y \rangle})$, where $\beta$ is as in \eqref{eq:D82cocycle}, give rise to equivalent module categories.
For $L=\langle xy, z \rangle$, it can be similarly computed that $\Omega_x$ restricts to a nontrivial cohomology class on $L$, and therefore the two pairs $(\langle xy, z \rangle, 1)$ and $(\langle xy, z \rangle, \beta|_{\langle xy,z \rangle})$ give rise to equivalent module categories.

\brk We  end by briefly remarking on the count of indecomposable bimodules for some examples, building on Example~\ref{ex:D8} (again, leaving bimodule conjugacy class computations to the reader).
For each example where every $L_{i,j}^{g_k}=\langle e \rangle$, the count of indecomposable bimodules does not change.
In fact, the only possibility where the count can change is when some $L_{i,j}^{g_k}$ has nontrivial Schur multiplier, and the only possibility for this in $D_8$ is when $L_{i,j}^{g_k}\cong \mathbb{Z}_2\times\mathbb{Z}_2$.
\end{example}

\section{Tensor algebras in group-theoretical fusion categories} \label{sec:tengt}

In this section, we study $\mathcal{C}$-tensor algebras [Definition~\ref{S,E},~\ref{def:minl}] for group-theoretical fusion categories $\mathcal{C}$ [Definition~\ref{def:grpthl}], building on the work in the previous section.  We begin by providing in Section~\ref{sec:grpthl} terminology and preliminary results for group-theoretical fusion categories $\mathcal{C}:=\mathcal{C}(G, \omega, K, \alpha)$, and then we recall in Section~\ref{sec:recon} the process of reconstructing a semisimple Hopf algebra whose representation category is tensor equivalent to $\mathcal{C}$. To obtain results on base algebras of $\mathcal{C}$-tensor algebras, we also examine indecomposable semisimple algebras in $\mathcal{C}$ in Section~\ref{sec:recon}. Finally, in Section~\ref{sec:H8}, we illustrate results by classifying all indecomposable semisimple algebras in  the category of finite-dimensional representations of the Kac-Paljutkin Hopf algebra $H_8$, up to Morita equivalence; this category is tensor equivalent to a group-theoretical fusion category $\mathcal{C}(D_8, \omega, \mathbb{Z}_2, 1)$.

\subsection{Background and notation} \label{sec:grpthl}
In this part we establish notation for group-theoretical fusion categories and  module categories over them. Recall the ${\sf Vec}_G^\omega$-module category $\mathcal{M}(K,\alpha)$ from Definition \ref{def:M(La,al)}.

\begin{definition}[$\mathcal{C}(G, \omega, K, \alpha)$] \label{def:grpthl}
A fusion category is called {\it group-theoretical} if it is categorically Morita equivalent to a pointed fusion category, that is, if it is equivalent to one of the form
$$ \mathcal{C}(G, \omega, K, \alpha) := (({\sf Vec}_G^\omega)^*_{\mathcal{M}(K,\alpha)})^{\rm op} $$
where $G$ is a finite group, $\omega$ is a 3-cocycle $G \times G \times G \to \kk^{\times}$, $K$ is a subgroup of $G$ with $\omega|_{K}$  trivial, and $\alpha$ is a 2-cochain $K \times K \to \kk^{\times}$ such that $d \alpha = \omega|_{K}$. 
\end{definition}

Thus, $\mathcal{C}(G, \omega, K, \alpha)$ is tensor equivalent to the category of $A(K,\alpha)$-bimodules in ${\sf Vec}_G^\omega$.

\begin{example} \label{ex:gt} The following are examples of group-theoretical fusion categories.
\begin{itemize}
\item[(1)] We have that $\mathcal{C}(G, \omega, \langle e \rangle, 1) \; \overset{\otimes}{ \sim }\;  {\sf Vec}_G^\omega$. Indeed, $A(\langle e \rangle,1) = {\bf 1}_{{\sf Vec}_G^\omega}$, so $$\mathcal{M}(\langle e \rangle, 1)  \; \overset{\otimes}{ \sim }\;  {\sf Mod}_{{\sf Vec}_G^\omega}\text{-}A(\langle e \rangle,1)  \; \overset{\otimes}{ \sim }\;  {\sf Vec}_G^\omega.$$ Moreover,  $\mathcal{C}^{\rm op}  \; \overset{\otimes}{ \sim }\;  \mathcal{C}_{\mathcal{C}}^*$. See  Examples~\ref{ex:S=1},~\ref{ex:extremes}, and~\ref{extreme-Gam}. 
\medskip
\item[(2)] We obtain that $\mathcal{C}(G, 1, G, 1)  \; \overset{\otimes}{ \sim }\; {\sf Rep}(G)$ as follows. First, $ {\sf Rep}(G)  \; \overset{\otimes}{ \sim }\; ({\sf Vec}_G)^*_{{\sf Vec}}$ [Remark~\ref{rmk:RepGam}]. By \cite[Corollary~3.4]{OstrikIMRN2003}, $\mathcal{M}\mathcal(G,1)$ has rank one.
Thus,
$$
({\sf Vec}_G)^*_{\mathcal{M}\mathcal(G,1)}
={\sf Fun}_{{\sf Vec}_G}(\mathcal{M}\mathcal(G,1), \mathcal{M}\mathcal(G,1))
\  \; \overset{\otimes}{ \sim }\; \ {\sf Fun}_{{\sf Vec}_G}({\sf Vec}, {\sf Vec}) 
=({\sf Vec}_G)^*_{{\sf Vec}}.
$$

\item[(3)] We also have that ${\sf Rep}(H)  \; \overset{\otimes}{ \sim }\;  \mathcal{C}(G, \omega, K, 1)$,  where $H$ is the bicrossed product $\kk^N\; {}^\tau\#_\sigma\; \kk K$. Here, $(K,N)$ is a matched pair of finite groups (so that $K$ and $N$ act on each other in a certain fashion) yielding a group $G = N \bowtie K$  that is a semi-direct product when either the action of $N$ on $K$, or $K$ on $N$, is trivial (this is also called an {\it exact factorization} of $G$). The maps $\sigma\colon K \times K \to (\kk^N)^\times$ and $\tau\colon N \times N \to (\kk^K)^{\times }$ are compatible cocycles defining the multiplication and comultiplication of $H$, respectively. Moreover, $\omega \in H^3(G, \kk^\times)$ represents the class $\overline{\omega}(\sigma, \tau)$ for the map $\overline{\omega}\colon \text{Opext}(\kk^N, \kk K) \to H^3(G, \kk^\times)$ in the {\it Kac sequence}. In this case, $H$ arises as the {\it abelian extension}
$$\hspace{.5in} \kk \to \kk^N \to H \to \kk K \to \kk.$$
See \cite[Section~3 and Proposition~4.5]{Natale2003}  for more details.
\end{itemize}
\end{example}

As a special case of (3) above we continue Example~\ref{ex:D8nontriv} below; we will consider this example in more detail in Section \ref{sec:H8}.

\begin{example} \label{ex:H8gt}
Take the groups $N = \langle x,y ~|~ x^2 = y^2 = e, \; xy=yx \rangle$ and $K = \langle z ~|~ z^2 = e \rangle$ from Example~\ref{ex:D8nontriv}, with the $N$-action on $K$ trivial, and the $K$-action on $N$ given by $z \cdot x = y$ and $z \cdot y = x$. Thus, $G = N \rtimes K \cong D_8$. With the cocycle $\omega = \omega(\sigma, \tau)$ given in \eqref{eq:D8nontriv}, we get that $\mathcal{C}(G, \omega, K, 1) \overset{\otimes}{ \sim } {\sf Rep}(H_8) $, for $H_8$ the Kac-Paljutkin semisimple Hopf algebra of dimension 8 \cite{KacPaljutkin, Kac1968}; see Definition~\ref{def:H8}.

Using Proposition~\ref{prop:Fun-Vec},  we can describe the simple objects of $\mathcal{C}(G, \omega, K, 1)$ in this case: 
$$X_0:= M(x, \rho_{\text{triv}}^{\langle e \rangle}), \quad X_1:= M(e, \rho_{\text{triv}}^K), \quad X_2:= M(e, \rho_{\text{sign}}^K), \quad X_3:= M(xy, \rho_{\text{triv}}^K), \quad X_4:= M(xy, \rho_{\text{sign}}^K).$$ 
Here, $M(g, \rho)$ is the simple object corresponding to the $K$-$K$ double coset $KgK$ in $G$, with $\rho$ an irreducible (projective) representation of $K \cap gKg^{-1}$ (with trivial Schur multiplier). Indeed, 
$$\{\langle z \rangle g \langle z \rangle\}_{g \in D_8} = \{e,z\} \cup \{x, y, xz, yz\} \cup \{xy, xyz\},$$
and we take representatives $g = e, x, xy$ and compute that $K \cap g K g^{-1}$ is $K$, $\langle e \rangle$, $K$, respectively.
\end{example}

From Theorem~\ref{thm:2categ}, we see that the following categories will play an essential role in studying group-theoretical fusion categories.

\begin{notation}[$\mathcal{M}^{K,\alpha}(L,\psi)$] \label{not:M(K,a,L,p)}
Fix a group $G$ and 3-cocycle $\omega$ on $G$.  We write
\begin{equation}\label{eq:MKaLb}
\M^{K,\alpha}(L,\psi):=\Fun_{{\sf Vec}_G^\omega}(\M(K,\alpha), \M(L,\psi)),
\end{equation}
which is an indecomposable semisimple left $\mathcal{C}(G, \omega, K, \alpha)$-module category by precomposition of functors.
\end{notation}

As a consequence of Theorem~\ref{thm:2categ}, Proposition \ref{prop:OstNat}, and Proposition~\ref{prop:Fun-Vec}, we have the following result.

\begin{proposition} \label{prop:modgrpthl}
Every indecomposable semisimple module category over the group-theoretical fusion category $\mathcal{C}(G, \omega, K, \alpha)$ is equivalent to some $\M^{K,\alpha}(L,\psi)$, where
\red{$(L, \psi)$ is as in} Proposition~\ref{prop:OstNat}, and its simple objects are $M(g, \rho)$ as in Proposition \ref{prop:Fun-Vec}.
The conditions for $\M^{K,\alpha}(L,\psi)$ and $\M^{K,\alpha}(L',\psi')$ to be equivalent are the same as in Proposition~\ref{prop:OstNat}, 
and the value \eqref{eq:arrow} is the rank of the functor category
 $\Fun_{\mathcal{C}(G, \omega, K, \alpha)}(\M^{K,\alpha}(L_i,\psi_i),\M^{K,\alpha}(L_j,\psi_j))$. \qed
\end{proposition}

\begin{example}\label{ex:CoK1mods}
Continuing Example \ref{ex:H8gt}, we saw in Example~\ref{ex:D8nontriv} that there are 6 equivalence classes of indecomposable semisimple ${\sf Vec}_{D_8}^\omega$-module categories, parametrized by pairs $(L, 1)$ where $L$ is one of the subgroups in~\eqref{eq:D8subgroups}.  Thus the indecomposable semisimple module categories over $\mathcal{C}(D_8, \omega, \langle z \rangle, 1) \overset{\otimes}{\sim} {\sf Rep}(H_8)$ are exactly those of the form $\M^{\langle z \rangle,1}(L,1)$ for $L$ one of the subgroups in \eqref{eq:D8subgroups}, up to equivalence.
\end{example}

\subsection{Reconstruction} \label{sec:recon}
Next, we turn our attention to group-theoretical fusion categories that are tensor equivalent to representation categories of semisimple Hopf algebras. 
We have the following reconstruction theorem for finite tensor categories $\mathcal{C}$ equipped with a fiber functor $F$. Recall that the $\kk$-algebra ${\sf End}(F)$ of functorial endomorphisms of $F$ has the structure of a finite-dimensional Hopf algebra; see \cite[Sections~1.8, 5.2, 5.3]{EGNO} for details.

\begin{theorem} \cite[Theorem~5.3.12]{EGNO} \label{thm:recon}
Consider a finite tensor (resp., fusion) category $\mathcal{C}$, and suppose that $\C$ admits a fiber functor $F$.  Then $\C$ is tensor equivalent to ${\sf Rep}(H)$, for $H = {\sf End}(F)$ a finite-dimensional (resp., semisimple)  Hopf algebra. \qed
\end{theorem}

We remind the reader that there are group-theoretical fusion categories that do not admit any fiber functor (e.g., ${\sf Vec}_G^\omega$ for $\omega$ non-trivial, \cite[Example~5.1.3]{EGNO}),
and furthermore there exist semisimple Hopf algebras whose representation categories are not group-theoretical \cite[Corollary~4.6]{Nikshych2008}.
In any case,  consider the result below.

\begin{proposition} \label{prop:gtfiber}
We have that a group-theoretical fusion category $\C:=\C(G, \omega, K, \alpha)$ admits a fiber functor if and only if there exists a subgroup $N \leq G$ and 2-cochain $\gamma$ on $N$ such that $d\gamma=\omega|_N$ where $G=KN$ and 
$K \cap N$ has a unique irreducible $\psi_{i,j}^e$-representation as in \eqref{eq:aijg}. Here, $\psi_i = \alpha$ and $\psi_j=\gamma$.
In this case, the fiber functor is
\begin{equation}
F_V: \C \to {\sf Vec}, \quad X \mapsto {\sf Hom}_{\mathcal{M}_0}(X \otimes V, V),
\end{equation}
where $\mathcal{M}_0 = \mathcal{M}^{K,\alpha}(N,\gamma)$ and $V$ is the unique simple object of $\M_0$. 
\end{proposition}
\begin{proof}
Since $\C$ admits a fiber functor if and only if it has a semisimple module category of rank 1, from Proposition~\ref{prop:modgrpthl} we see that this occurs if and only if there exists $(N,\gamma)$ as in the statement with $\mathcal{M}^{K,\alpha}(N,\gamma)$ of rank 1.  This rank is counted in Proposition \ref{prop:Fun-Vec}.
\end{proof}

Applying Theorem \ref{thm:recon}, this gives a criterion for $\C(G,\omega,K,\alpha)$ to be equivalent to the representation category of a Hopf algebra. We fix notation for this situation, which will be studied in more detail for the remainder of the paper.

\begin{notation}[$\C$, $F_V$, $\mathcal{M}_0$, $H(N,\gamma)$] \label{not:H(N,gamma)} 
Assume that $\C:=\mathcal{C}(G, \omega, K, \alpha)$ admits a fiber functor, and fix one $F_V \colon \C \to \sVec$ as in Proposition~\ref{prop:gtfiber}. Take also $\mathcal{M}_0 := \mathcal{M}^{K,\alpha}(N,\gamma)$ as in Proposition \ref{prop:gtfiber}. We write $H(N,\gamma)$ for the corresponding semisimple Hopf algebra obtained from the fiber functor $F_V$.  Note that the data $(G,\omega,K,\alpha)$ defining $H(N,\gamma)$ is understood from context.  Thus, we have
${\sf Rep}(H(N,\gamma)) \overset{\otimes}{\sim} \mathcal{C}(G, \omega, K, \alpha)$.
\end{notation}

Our next goal is to use the classification of indecomposable semisimple algebras in $\C$ from previous sections to study indecomposable semisimple algebras in ${\sf Rep}(H(N,\gamma))$ via the equivalence ${\sf Rep}(H(N,\gamma)) \overset{\otimes}{\sim}\C$.
We refer to such algebras as {\it indecomposable semisimple $H(N,\gamma)$-algebras}.
To describe these algebras more explicitly, we recall the internal End construction.

\begin{definition}[$\underline{\sf End }(M)$] \cite[Section~7.8]{EGNO}
Let $\mathcal{M}$ be a $\C$-module category and fix an object $M \in \mathcal{M}$. The {\it internal End of $M$} is the object in $\mathcal{C}$ that represents the functor $\mathcal{C} \to {\sf Vec}$, $X \mapsto {\sf Hom}_{\mathcal{M}}(X  \otimes M, M)$; it is denoted by $\underline{\sf End }(M)$. Namely, we get that 
\begin{equation}\label{eq:internalend}
{\sf Hom}_{\mathcal{M}}(X  \otimes M, M) \; \cong \;  {\sf Hom}_{\mathcal{C}}(X, \underline{\sf End }(M)).
\end{equation}
 \end{definition}
 
 We have that $\underline{\sf End }(M)$ is an algebra in $\mathcal{C}$ (see \cite[Section~7.9]{EGNO}).

 \begin{example} \label{ex:Enddual}
For any finite-dimensional Hopf algebra $H$, we have $\EEnd(\kk) \cong  H^*$ as algebras in $\Rep(H)$  (see \cite[Example~7.9.11]{EGNO}).
Applying this to $H = {\sf End}(F_V)$ using the equivalence of Proposition~\ref{prop:gtfiber},  we get that $F_V(\EEnd(V))= H^*$ in $\Rep(H)$, for the simple object $V \in \mathcal{M}_0$.
 \end{example}

\begin{lemma}[$A_M$, $A_{M(g,\rho)}$] \label{lem:AM}
Every indecomposable semisimple $H(N,\gamma)$-algebra is Morita-equivalent in ${\sf Rep}(H(N,\gamma))$ to an algebra $F_V(A_M)$ where $A_M:= A_{M(g, \rho)} :=\EEnd(M)$, with $M = M(g,\rho)$ a simple object of some $\M^{K,\alpha}(L,\psi)$ as in Proposition \ref{prop:modgrpthl}.
 
Furthermore, we have the dimension calculation:
\begin{equation}\label{eq:dimFVAM}
\dim_\kk F_V(A_M) = \sum_{X \in \textnormal{Irr}(\C)} m_X(M)\; \FPdim_{\mathcal{C}} X
\quad \quad \text{where }\; m_X(M) := \dim_\kk {\sf Hom}_{\mathcal{M}}(X \otimes M, M).
\end{equation}

\end{lemma}

\begin{proof}
From the tensor equivalence ${\sf Rep}(H(N,\gamma)) \overset{\otimes}{\sim} \mathcal{C}$ and Theorem \ref{thm:ModcA}, we see that indecomposable semisimple $H(N,\gamma)$-algebras are in bijection with indecomposable semisimple $\mathcal{C}$-module categories, which are all of the form $\M^{K,\alpha}(L,\psi)$ by Proposition~\ref{prop:modgrpthl}.
Then from \cite[Theorem~7.10.1]{EGNO} we have that an indecomposable semisimple module category $\M^{K,\alpha}(L,\psi)$ over $\mathcal{C}$ is equivalent to the category of $A_M$-modules in $\mathcal{C}$, where $M\in \M^{K,\alpha}(L,\psi)$ can be taken to be any nonzero object.

For the dimension calculation, we can decompose $A_M$ as a direct sum of irreducibles in $\mathcal{C}$ to write
\begin{equation} \label{eq:AM}
A_M = \bigoplus_{X \in \textnormal{Irr}(\mathcal{C})} m_X(M) \; X.
\end{equation}
Here we are using that $m_X(M)=\dim_\kk \Hom_{\mathcal{C}}(X, A_M)$ by \eqref{eq:internalend}, 
which gives the multiplicity of $X$ in $A_M$ since $\C$ is semisimple.  Then the dimension calculation follows from \cite[Proposition~4.5.7]{EGNO}: we obtain that  $\dim_\kk F_V(X) = \FPdim_{{\sf Vec}} F_V(X) = \FPdim_{\mathcal{C}} X$.
\end{proof}

\begin{lemma} \label{lem:corresp} There is a bijection between the collection of simple modules over the $\kk$-algebra $F_V(A_M)$ and the collection of simple objects of the category ${\sf Fun}_{\mathcal{C}}(\mathcal{M}_0, \mathcal{M}^{K,\alpha}(L,\psi))$. 
\end{lemma}
\begin{proof}
Let $B:=\EEnd(V)$, so we have $\M_0\sim\Mod_\C$-$B$ and recall that $\mathcal{M}^{K,\alpha}(L,\psi))\sim \Mod_\C$-$A_M$.  Then ${\sf Fun}_{\mathcal{C}}(\mathcal{M}_0, \mathcal{M}^{K,\alpha}(L,\psi))$ is equivalent to ${\sf Bimod}_\C(B, A_M)$ \cite[Proposition~7.11.1]{EGNO}, which can be identified with the category of right $A_M$-modules in the category of left $B$-modules in $\C$, the latter of which is identified with ${\sf Vec}$ via the functor $F_V$. Thus, ${\sf Fun}_{\mathcal{C}}(\mathcal{M}_0, \mathcal{M}^{K,\alpha}(L,\psi))$ is identified with the category of right $F_V(A_M)$-modules in ${\sf Vec}$ as desired.
\end{proof}

\begin{remark}
Suppose that $P$ is a simple object in ${\sf Mod}_\kk (F_V(A_M))$ with corresponding simple object $P'$ of ${\sf Irr}({\sf Fun}_{\mathcal{C}}(\mathcal{M}_0, \mathcal{M}^{K,\alpha}(L,\psi)))$ as in the lemma above. Then, $$\dim_\kk P= \sqrt{\frac{\dim_\kk F_V(A_M)}{\textnormal{FPdim } \mathcal{C}}} \; \textnormal{FPdim} \; P'.$$ 
Indeed, consider the action of $\mathcal{C}_{\mathcal{M}_0}^*$ on ${\sf Fun}_{\mathcal{C}}(\mathcal{M}_0, \mathcal{M}^{K,\alpha}(L,\psi))$ via precomposition. Here, $\mathcal{C}_{\mathcal{M}_0}^*$ is the representation category of the dual Hopf algebra $H(N,\gamma)^*$; see Example~\ref{ex:Enddual}.
Now for any $H^*$-module $X$,
we have that $$\dim_\kk (X\otimes P)=(\dim_\kk X)(\dim_\kk P)=({\sf FPdim}\; X)(\dim_\kk P).$$
This means that there exists a positive number $\lambda$ such that $\dim_\kk P= \lambda ({\sf FPdim}\; P')$, as a function satisfying the displayed equality above is a Frobenius-Perron eigenvector and thus is unique up to scaling. We get the value $\lambda$ taking the sum of squares of the last equation.\end{remark}


\subsection{${\sf Rep}(H_8)$-tensor algebras}  \label{sec:H8}
We now consider Example \ref{ex:H8gt} in more detail.
Consider the groups \linebreak $N=\langle x,y ~|~ x^2= y^2 = e, \; xy=yx\rangle$ and $K=\langle z ~|~ z^2 =e\rangle$ from Example~\ref{ex:D8nontriv} with the $N$-action on $K$ trivial, and the $K$-action on $N$ given by $z \cdot x = y$ and $z \cdot y = x$. Thus, $G = N \rtimes K \cong D_8$.  Taking the cocycle $\omega$ of \eqref{eq:D8nontriv}, we obtain the group-theoretical fusion category $\mathcal{C}(G, \omega, K, 1)$.

To obtain a Hopf algebra, we construct a fiber functor as in Section \ref{sec:recon} by letting $\M_0 = \M^{K,1}(N,\mu)$ with $\mu$ the nontrivial 2-cocycle on $N$ of \eqref{eq:Z2xZ2}. Indeed, $d \mu = \omega|_N$ as 
\[
\begin{array}{rl}
d \mu (x^{i_1}y^{j_1},~ x^{i_2}y^{j_2}, ~x^{i_3}y^{j_3}) 
&= \mu(x^{i_1}y^{j_1},~ x^{i_2+i_3}y^{j_2+j_3})\;
\mu(x^{i_2}y^{j_2}, ~x^{i_3}y^{j_3}) \\
\medskip
& \quad \mu^{-1}(x^{i_1+i_2}y^{j_1+j_2}, ~x^{i_3}y^{j_3}) \; \mu^{-1}(x^{i_1}y^{j_1},~ x^{i_2}y^{j_2})\\
& = (-1)^{j_1(i_2+i_3)+j_2i_3 - (j_1+j_2)i_3 - j_1 i_2} = 1,
\end{array}
\]
and $\omega|_N$ is trivial.
The resulting Hopf algebra is $H(N,\mu)=H_8$, the unique semisimple, noncommutative, noncocommutative Hopf algebra of dimension~8 (up to isomorphism), and we have $\mathcal{C}(G, \omega, K, 1) \overset{\otimes}{\sim} {\sf Rep}(H_8)$.
See \cite[Theorem~2.13]{Masuoka:dim6-8} or \cite[Section~16.3]{Radford} for the presentation of $H_8$ below.
 
 \begin{definition} \label{def:H8}
\cite{Kac1968}
The {\it Kac-Paljutkin Hopf algebra} $H_8$ is defined by generators $x, y, z$ subject to relations 
 $$\textstyle x^2=1, \quad y^2= 1, \quad xy=yx, \quad zx=yz, \quad zy=xz, \quad z^2 = \frac{1}{2}(1 + x+ y-xy),$$ where $x$ and $y$ are grouplike elements, and 
 $$\Delta(z) =\textstyle \frac{1}{2}(1 \otimes 1 + y \otimes 1 + 1 \otimes x - y \otimes x)(z \otimes z), \quad \epsilon(z) =1, \quad S(z) = z.$$ 
 \end{definition} 
 
Recall that $H_8$ has 5 isomorphism classes of irreducible representations, which have the following explicit descriptions (see, for example, \cite[p.~530]{Radford}):
\begin{equation}\label{eq:H8reps}
\begin{array}{l}
\medskip
W_0 := \kk^2,\qquad x\cdot(v_1,v_2)=(-v_1,v_2),\quad y\cdot (v_1,v_2)=(v_1,-v_2), \quad z\cdot (v_1,v_2)=(v_2,v_1);\\
W_1:=\kk _{1,1,1}; \qquad \quad W_2:=\kk _{1,1,-1}; \qquad  \quad W_3:=\kk _{-1,-1,\sqrt{-1}} ; \qquad \quad W_4:=\kk _{-1,-1,-\sqrt{-1}} ;
\end{array}
\end{equation}
where $\kk_{\lambda,\lambda',\lambda''}$ denotes the 1-dimensional representation with $x$-, $y$,- $z$-action being scalar multiplication by $\lambda$, $\lambda'$, $\lambda''$, respectively.

\begin{lemma}\label{lem:H8match}
Via the equivalence $\mathcal{C}(G, \omega, K, 1) \overset{\otimes}{\sim} {\sf Rep}(H_8)$, each simple object $X_i \in \C(G, \omega,K,1)$ from Example~\ref{ex:H8gt} corresponds to $W_i \in \textnormal{Irr}(\Rep(H_8))$ in \eqref{eq:H8reps}.
\end{lemma}

\begin{proof}
It is clear that the 2-dimensional objects should be in correspondence, so $W_0$ corresponds to $X_0$. Also the unit object is $W_1$ in one realization and $X_1$ in the other, so they correspond. Now, $W_3$, $W_4$ are permuted by complex conjugation, but $X_2$ is clearly fixed by conjugation, so $W_3$, $W_4$ must correspond to $X_3$, $X_4$. Note that whether $W_3$ corresponds to $X_3$ or $X_4$ depends on the choice of the 3-cocycle $\omega$ (whether we use $\sqrt{-1}$ or $-\sqrt{-1}$ in its formula), but we make a choice for $\omega$ so that $W_3$ corresponds to $X_3$ and $W_4$ to $X_4$. 
\end{proof}

\brk
We now classify indecomposable semisimple algebras in $\Rep(H_8)$ using the general theory developed above.  
We start by proving that the following list of semisimple $H_8$-module algebras is a classification, up to Morita equivalence, then afterwards explain how to obtain them from the tensor categorical approach.  The decomposition of each algebra as an $H_8$-module is noted for future reference; these can be directly computed.

\medskip

\begin{enumerate}[(i)]
\item $S=\kk$ with $x, y, z$ acting as the identity, so $S\cong W_1$.
\smallskip
\item $S=\kk^2$ with $x, y$ acting as the identity and $z\cdot (a, b)=(b,a)$  for all $(a,b) \in \kk^2$, so $S\cong W_1\oplus W_2$.
\smallskip
\item $S=\kk^2$ with 
\begin{equation*}
x\cdot (a, b)=y\cdot (a, b)=(b,a), \qquad z\cdot (a, b)=(a\theta+b\bar{\theta},\,b\theta+a\bar{\theta})
\end{equation*}
where $\theta=\frac{1}{2}(1+i)$ and $\bar{\theta}=\frac{1}{2}(1-i)$, so $S\cong W_1\oplus W_3$.
\smallskip
\item $S=\kk^4$ with
\begin{equation*}
x\cdot (a, b, c, d)=(a,b, d, c), \qquad y\cdot (a, b, c, d)=(b, a, c, d), \qquad z\cdot (a, b, c, d)=(c, d, a, b),
\end{equation*}
so $S\cong W_0\oplus W_1 \oplus W_2$.
\item $S=\kk^4$ with
\begin{equation*}
x\cdot (a, b, c, d)= y\cdot (a, b, c, d)=(b, a, d, c), \qquad z\cdot (a, b, c, d)=(c\theta+d\bar{\theta},\,d\theta+c\bar{\theta},\, a\theta+b\bar{\theta},\,b\theta+a\bar{\theta})
\end{equation*}
where $\theta=\frac{1}{2}(1+i)$ and $\bar{\theta}=\frac{1}{2}(1-i)$, as above, so $S\cong W_1\oplus W_2\oplus W_3 \oplus W_4$.
\medskip
\item $S = H_8^*$ with action $(h \cdot f)(t) = f(S(h)t)$ for $h,t \in H_8$ and $f \in H_8^*$, so $S\cong W_0^{\oplus 2} \oplus W_1\oplus W_2\oplus W_3 \oplus W_4$.
\end{enumerate}

\medskip

\begin{theorem} \label{prop:H8Morita}
The algebras \textnormal{(i)}--\textnormal{(vi)} above are pairwise Morita inequivalent in the tensor category $\Rep(H_8)$, and every indecomposable semisimple algebra in ${\sf Rep}(H_8)$ is Morita equivalent to one on this list.
\end{theorem}
\begin{proof} 
Using Theorem \ref{thm:param}, Example~\ref{ex:D8nontriv}, and Proposition~\ref{prop:modgrpthl}, we see that there are 6 equivalence classes of indecomposable semisimple $H_8$-module algebras (which serve as base algebras of ${\sf Rep}(H_8)$-tensor algebras).  Since we have produced a list of 6 such algebras, it is enough to show that the algebras in the list are pairwise Morita inequivalent.

Note that if two algebras $S_1$ and $S_2$ are Morita equivalent as $H$-module algebras, then they are also Morita equivalent as $H'$-module algebras, for any Hopf subalgebra $H'$ of $H$. In particular, by taking $H' = \kk$, we have by the Artin-Wedderburn theorem that it suffices to verify that the algebras $S$ in (ii) and (iii), and that the algebras $S$ in (iv) and (v), are Morita inequivalent. For the former, take $H'$ to be the Hopf subalgebra $\langle x \rangle$ generated by $x$; observe that $S$ in (ii) is decomposable in ${\sf Rep}(\langle x \rangle)$, whereas $S$ in (iii) is not. For the latter, observe that $S$ in (iv) has three indecomposable summands in ${\sf Rep}(\langle x \rangle)$, whereas $S$ in (v) only has two indecomposable summands. So, the result holds.
\end{proof}

The following example illustrates our classification.

\begin{example}
Consider the adjoint action of $H_8$ on itself, recalling the algebra decomposition $H_8 \cong \kk^4 \oplus {\sf Mat}_2(\kk)$.
The ideal $S={\sf Mat}_2(\kk) \subset H_8$ is $H_8$-stable with action
\begin{equation*}
x\cdot \begin{bmatrix}a & b\\c & d\end{bmatrix}=\begin{bmatrix}a & -b\\-c & d\end{bmatrix}, \qquad
y\cdot \begin{bmatrix}a & b\\c & d\end{bmatrix}=\begin{bmatrix}a & -b\\-c & d\end{bmatrix}, \qquad
z\cdot \begin{bmatrix}a & b\\c & d\end{bmatrix}=\begin{bmatrix}d & c\\-b & a\end{bmatrix},
\end{equation*}
and is an indecomposable semisimple algebra in ${\sf Rep}(H_8)$ which is not on our list above.  However, using  Example~\ref{ex:S=1} we see it is Morita equivalent to the algebra in (i) above via ${\sf Mat}_2(\kk) \cong W_0 \otimes W_0^*$ in $\Rep(H_8)$, where $W_0$ is as in~\eqref{eq:H8reps}.
\end{example}

\begin{remark}
In the algebra (iii) we see that the trace of the action of $z$ is $2\theta=1+i$.  Therefore if we also consider the $H_8$-module algebra obtained by switching the roles of $\theta, \bar\theta$ (complex conjugation), it has the same dimension as the algebra in (iii), but they are not isomorphic as $H_8$-module algebras.  They are Morita equivalent, however, by the same \red{reasoning as in} the proof of Theorem \ref{prop:H8Morita}.  
\end{remark}

Recall from Example \ref{ex:CoK1mods} that the indecomposable semisimple module categories over (and thus indecomposable algebras in) $\mathcal{C}(G, \omega, K, 1)$ are in bijection with the subgroups $L \leq G$ of \eqref{eq:D8subgroups}. We now show how to match these subgroups with the algebras (i)--(vi).

\begin{proposition}\label{prop:H8match}
The following correspondence matches the indecomposable semisimple algebras in \linebreak $\C(G,\omega,K,1)$ and indecomposable semisimple $H_8$-algebras.  That is, for each pairing $L\leftrightarrow S$ indicated below, we have $F_V(A_M) \cong S$ in $\Rep(H_8)$, where $M \in \M^{K,1}(L,1)$ is any simple object.
\begin{equation}\label{eq:match}
\langle e \rangle \leftrightarrow \textnormal{(ii)} \qquad
\langle x \rangle \leftrightarrow \textnormal{(iv)} \qquad
\langle xy \rangle \leftrightarrow \textnormal{(v)} \qquad
\langle z \rangle \leftrightarrow \textnormal{(i)} \qquad
\langle x,y \rangle \leftrightarrow \textnormal{(vi)} \qquad
\langle xy,z \rangle \leftrightarrow \textnormal{(iii)}
\end{equation}
\end{proposition}
\begin{proof}
The multiplicity of each irreducible representation of $H_8$ in the algebras (i)--(vi) has already been described above.  We will make the match by computing the multiplicity $m_X(M)$ of each $X \in \textnormal{Irr}(\mathcal{C}(G, \omega, K, \alpha))$ in $A_M$ using \eqref{eq:dimFVAM}, and comparing these values using Lemma \ref{lem:H8match}. In each case, take $M = M(KeL, \rho_\text{triv}^{K \cap L})$.

\brk
For $L=\langle e\rangle$, we have that $K \cap L = \langle e \rangle$. Moreover $X_0 \otimes M$ is supported on $K x L \cup K y L$; so, $m_{X_0}(M) = 0$. For $i = 1,2$, we have that $X_i \otimes M$ is supported on $K e L$; so, $m_{X_i}(M) = 1$. Lastly, for $j = 3,4$, we get that  $X_j \otimes M$ is supported on $K xy L$; so, $m_{X_j}(M) = 0$. Therefore, $A_M = X_1 \oplus X_2$ as an object in $\mathcal{C}$ in this case, thus it matches with (ii) where $S\cong W_1\oplus W_2$ in $\Rep(H_8)$.

\brk
On the other hand, take $L=\langle xy, z \rangle$ and we get $K \cap L =K$. Again, $X_0 \otimes M$ is supported on $K x L \cup K y L$; so, $m_{X_0}(M) = 0$. For $i = 1,2, 3, 4$, we have that $X_i \otimes M$ is supported on $K e L$, but the 2-cocycle on $K \cap L$ is trivial if and only if $i =1,3$ here. Therefore, $A_M = X_1 \oplus X_3$ as an object in $\mathcal{C}$ in this case, thus it matches with (iii) where $S\cong W_1\oplus W_3$ in $\Rep(H_8)$.

\brk
The 4 remaining matchings of the algebras $S$ with the subgroups $L$ are computed similarly.
\end{proof}

\begin{remark}
The $H_8$-module algebras (i)--(vi) were originally obtained by a variety of \emph{ad hoc} methods including hand computation and Maple code, and we were able to match them with the pairs $(L,\psi)$ arising in Proposition~\ref{prop:OstNat} after the fact.
We note that all the algebras in our list occur as coideal subalgebras of $H_8^*$ (see the lattice diagram in \cite[Figure~1]{DT11}).  In that diagram, the only Morita equivalences as $H_8$-module algebras are between I1 and I2, and between J2 and J4.

For a semisimple Hopf algebra $H$ in general, an indecomposable semisimple $H$-module algebra $S$ is isomorphic as an $H$-module algebra to a coideal subalgebra of $H^*$ if and only if $S$ has a 1-dimensional representation, which is why all 6 algebras for $H_8$ come as coideal subalgebras. However, different coideal subalgebras can be Morita equivalent and even isomorphic as $H$-module algebras.  In fact, every 1-dimensional representation of such an algebra $S$ gives a realization of $S$ as a coideal subalgebra, by composing the coaction map with this representation; such coideal subalgebras may or may not be the same for different 1-dimensional representations.  We leave this observation as a starting point for further investigation; cf. \cite[Lemma~3.9]{EtingofWalton}.
\end{remark}

\brk
Regarding indecomposable bimodules, Example \ref{ex:D8nontriv} yields examples  of such bimodules via  Theorem \ref{thm:2categ} and the equivalence ${\sf Rep}(H_8) \; \overset{\otimes}{\sim} \; \mathcal{C}(G, \omega, K, 1)$.  While we have explicit formulas in examples for actions of $H_8$ on some bimodules, we leave the systematic study of this to future work.



\section{Path algebras in group-theoretical fusion categories} \label{sec:Path}
In this section, we return to one of the original motivations of this work and classify path algebras that admit a grade-preserving action of a semisimple Hopf algebra $H$. Here,  we restrict our attention to such $H$ whose representation category is group-theoretical [Definition~\ref{def:grpthl}]. 

\brk
For now, let us fix $\mathcal{C} = (\mathcal{C}, F)$ a finite tensor category equipped with a fiber functor $F: \C \to {\sf Vec}$, and begin by defining below a $\mathcal{C}$-path algebra, which is a special type of $\mathcal{C}$-tensor algebra [Definition~\ref{S,E}].

\begin{definition} \label{def:Cpath}
We say that a $\mathcal{C}$-tensor algebra $T_S(E)$ is a {\it $\mathcal{C}$-path algebra} if $F(S)$ is a commutative $\kk$-algebra. In this case, we say that $S$ is {\it $\kk$-commutative}, for short.
\end{definition}

The notion of whether an exact algebra $S$ in $\mathcal{C}$ is $\kk$-commutative depends on its Morita equivalence class in $\C$, but it does not depend on the choice of $F$: indeed, we have  $\dim_\kk F(S) = \FPdim_{{\sf Vec}} F(S) = \FPdim_{\mathcal{C}} S$ \cite[Proposition~4.5.7]{EGNO}, and computing this dimension is key to checking this property for a given algebra.  The terminology is motivated as follows. 

\begin{remark} \label{rem:comm}
If the base algebra $S$ of $T_S(E)$ is $\kk$-commutative, then $F(S)$ is a semisimple, finite-dimensional commutative $\kk$-algebra, and thus, is a product of fields. In this case, $F(S)$ can be realized as the path algebra $\kk Q_0$ on a \red{finite set of vertices} $Q_0$.  By choosing an appropriate basis $Q_1$ of the generating bimodule $E$, we can construct a quiver $Q=(Q_0, Q_1)$ whose path algebra inherits a Hopf action of $H:=\End(F)$ and $T_S(E) \cong \kk Q$ as $H$-module algebras.
\end{remark}

For $\mathcal{C}=\mathcal{C}(G, \omega, K, \alpha)$ a group-theoretical fusion category equipped with a fiber functor, a condition when a $\C$-tensor algebra is a $\C$-path algebra is established in Corollary~\ref{cor:pathgt} in Section~\ref{sec:Cpath}. This is achieved by studying when the algebras $A_M$ from Lemma~\ref{lem:AM} are $\kk$-commutative [Proposition~\ref{prop:inequality}]. Then, the special case when the group $G$ has an exact factorization is examined in Section~\ref{sec:exactfact}; we end that section by continuing Example~\ref{ex:H8gt} and the work of Section~\ref{sec:H8} for ${\sf Rep}(H_8)$.

\subsection{Indecomposable $\kk$-commutative algebras in $\mathcal{C}(G, \omega, K, \alpha)$} \label{sec:Cpath} We continue as in Notation \ref{not:H(N,gamma)} to fix a group-theoretical fusion category $\mathcal{C}:=\mathcal{C}(G, \omega, K, \alpha)$ equipped with fiber functor $F_V$, so that it is tensor equivalent to ${\sf Rep}(H(N,\gamma))$. We also fix $A_M$ for $M =M(g,\rho) \in \M:=\mathcal{M}^{K,\alpha}(L,\psi)$ as in Lemma~\ref{lem:AM}.

\brk
The results of this section require use of Frobenius-Perron dimension, which can be reviewed in \cite[Chapters~3 and~6]{EGNO}.
Recall that the \emph{regular objects} of $\mathcal{C}$ \cite[Definition~6.1.6]{EGNO} and $\mathcal{M}$ are
\begin{equation}\label{eq:regularobjects}
\red{R_{\mathcal{C}} =  \sum_{X \in \textnormal{Irr}(\mathcal{C}) } (\FPdim_{\mathcal{C}}X)  X, \qquad R_{\mathcal{M}} = \sum_{Z \in \textnormal{Irr}(\mathcal{M}) } (\FPdim_{\mathcal{M}}Z)  Z,}
\end{equation}
where we use the canonical normalization to define $\FPdim_\mathcal{M}$, meaning $\FPdim_{\mathcal{M}} R_\mathcal{M} = \FPdim_\mathcal{C} R_\mathcal{C}$ \cite[Exercise~7.16.8]{EGNO}. \red{Note that these regular objects lie in the Grothendieck groups of $\mathcal{C}$ and $\mathcal{M}$, respectively, i.e. they are virtual objects rather than actual ones (e.g. FPdim$_\mathcal{M}$ may not be an integer so these are not well defined objects in general).}

\begin{lemma} \label{lem:FPdimM}
Let $M_i=M(g_i,\rho_i)$ be the simple objects of $\mathcal{M}$. Then we have for all $i$ that
\begin{equation}
{\sf FPdim}_\mathcal{M} M_i = \frac{\sqrt{|K| \; |L|}}{|K\cap g_iLg_i^{-1}|} \; \dim_\kk \rho_i.
\end{equation}
\end{lemma}
\begin{proof}
There exists a positive number $\lambda$ such that ${\sf FPdim}_\mathcal{M} M_i=\lambda (\dim_\kk M_i)$ for all $i$, since both ${\sf FPdim}_\mathcal{M} M_i$ and $\dim_\kk M_i$ are Frobenius-Perron eigenvectors of multiplication by $X \in \mathcal{C}$, and such an eigenvector is unique up to scaling. 
Thus, $${\sf FPdim}_\mathcal{M} M_i = \lambda \; \frac{|K| \; |L|}{|K\cap g_iLg_i^{-1}|} \; \dim_\kk \rho_i.$$ So, summing the squares of these dimensions over all $i$, we get $|G|=\lambda^2 \; |K|\; |L|\; |G|$, which yields \linebreak $\lambda=(|K|\; |L|)^{-1/2}$. 
Thus, ${\sf FPdim}_\mathcal{M} M_i = ((|K||L|)^{1/2}/|K\cap g_iLg_i^{-1}|) \dim \rho_i$.
\end{proof}

\begin{lemma}\label{lem:dimFVAM}
For the object $M$ above, we have that $\dim_\kk F_V(A_M)=({\sf FPdim}_\mathcal{M} M)^2$.
\end{lemma}
\begin{proof}
We have that $R_{\mathcal{C}} \otimes M$ is an eigenvector in $\mathrm{Gr}(\mathcal{M})$ for the left action of any $X \in \mathcal{C}$, thus it must be a scalar multiple of $R_\mathcal{M}$ \cite[Proposition~8.5]{ENO}.
Since $\FPdim_\mathcal{M}(R_{\mathcal{C}} \otimes M) = (\FPdim_\mathcal{C}R_\mathcal{C})(\FPdim_\mathcal{M}M)$, the canonical normalization condition above gives $R_{\mathcal{C}} \otimes M = (\FPdim_\mathcal{M} M) R_{\mathcal{M}}$.
From \eqref{eq:AM} we see that \eqref{eq:dimFVAM} can be rewritten as $\dim_\kk F_V(A_M)=\dim_\kk\Hom_\C(R_\C, A_M)$, then \eqref{eq:internalend} gives the first equality of
\[
\dim_\kk F_V(A_M)=\dim_\kk {\sf Hom}_{\mathcal{M}}(R_{\mathcal{C}} \otimes M,  M)=\dim_\kk {\sf Hom}_{\mathcal{M}}((\FPdim_\mathcal{M} M) R_{\mathcal{M}},  M)=(\FPdim_\mathcal{M} M)^2. \qedhere
\]
\end{proof}

Next, we establish necessary and sufficient conditions for  $A_M$ to be $\kk$-commutative.

\begin{lemma} \label{lem:r(NgLp)}
The rank of ${\sf Fun}_{\mathcal{C}}(\M_0, \mathcal{M}^{K,\alpha}(L,\psi))$ is less than or equal to $\dim_\kk F_V(A_M)$, with equality if and only if $A_M$ is $\kk$-commutative.
\end{lemma}

\begin{proof}
From Lemma~\ref{lem:corresp}, the simple objects of ${\sf Fun}_{\mathcal{C}}(\M_0, \mathcal{M}^{K,\alpha}(L,\psi))$
are in bijection with simple \red{$F_V(A_M)$-modules}.
Since the $\kk$-algebra $F_V(A_M)$ is semisimple, the statement follows from the Artin-Wedderburn theorem, recalling that a semisimple $\kk$-algebra is commutative if and only if  each of its simple modules is 1-dimensional.
\end{proof}

\begin{proposition} \label{prop:inequality}  Let $m_{\gamma, \psi}(h)$ denote the \red{number} of irreducible projective representations of the group $N \cap hLh^{-1}$ with Schur multiplier $\gamma\psi^{-1}_h$ from~\eqref{eq:aijg}. Then we have that 
\begin{equation} \label{eq:laststep}
\sum_{h \in N \backslash G / L} m_{\gamma, \psi}(h) ~\leq~ \frac{|K| \; |L|}{|K \cap gLg^{-1}|^2}(\dim \rho)^2.
\end{equation}
This is an equality if and only if the algebra $A_M$ in $\mathcal{C}$ is $\kk$-commutative.
\end{proposition}

\begin{proof}
By Proposition \ref{prop:modgrpthl}, \red{the left hand side} is equal to the rank of ${\sf Fun}_{\mathcal{C}}(\M_0, \mathcal{M}^{K,\alpha}(L,\psi))$. 
By Lemmas~\ref{lem:FPdimM} and~\ref{lem:dimFVAM}, the right hand side is equal to  $\dim_\kk F_V(A_M)$, so this proposition is equivalent to Lemma~\ref{lem:r(NgLp)}.
\end{proof}

Now we can determine when a $\mathcal{C}(G, \omega, K, \alpha)$-tensor algebra is equivalent to a $\mathcal{C}(G, \omega, K, \alpha)$-path algebra in the sense of Definition~\ref{def:Cpath}.

\begin{corollary} \label{cor:pathgt}   Take $\mathcal{C}:= \mathcal{C}(G, \omega, K, \alpha)$ a group-theoretical fusion category equipped with a fiber functor, and let $T_S(E)$  be a $\mathcal{C}$-tensor algebra. Consider the decomposition of $S$ into a direct sum of indecomposable  semisimple algebras $A_i$ in $\mathcal{C}$. Then,
$T_S(E)$ is equivalent to a $\mathcal{C}$-path algebra if and only if each $A_i$ is Morita equivalent to an algebra of the form $A_{M(g, \rho)}$ for which equality holds in~\eqref{eq:laststep}.  \qed
\end{corollary}

\subsection{Exact factorization case} \label{sec:exactfact} Keep the setting of the previous subsection. Below we also assume that $G = KN$ is an exact factorization, that is, we have  $|G| = |K| \; |N|$.

\begin{lemma} \label{lem:exactfact} Suppose that $G = KN$ is an exact factorization. Then, for any $g \in G$ and any subgroup $L$ of $G$, we get 
$$\sum_{h \in N \backslash G / L} |N \cap hLh^{-1}| ~\leq~ \frac{|K| \; |L|}{|K \cap gLg^{-1}|^2}.$$ 
Moreover, this is an equality if and only if $|N \cap hLh^{-1}| \cdot |K \cap gLg^{-1}| ~=~ |L|$
 for each $h \in G$.
\end{lemma}

\begin{proof}
We can rewrite the inequality in question as 
$$\sum_{h \in  G} \frac{|N \cap hLh^{-1}|^2}{|N| \cdot |L|} ~\leq~ \frac{|K| \; |L|}{|K \cap gLg^{-1}|^2}$$ 
or 
$$\sum_{h \in  G} |N \cap hLh^{-1}|^2  \cdot |K \cap gLg^{-1}|^2 ~\leq~ |G| \; |L|^2.$$
So, it suffices to show that for each $h \in G$ we have $|N \cap hLh^{-1}|  \cdot |K \cap gLg^{-1}|  \; \leq \; |L|.$ But this holds because $N \cap hLh^{-1} = h^{-1}N h \cap L$, and $K \cap gLg^{-1} = g^{-1}Kg \cap L$, and $G = (g^{-1}Kg)(h^{-1}N h)$ is an exact factorization. We also get that the inequality in the claim is an equality if and only if $|N \cap hLh^{-1}|  \cdot |K \cap gLg^{-1}|  \; = \; |L|$  for each $h \in G$. 
\end{proof}

\begin{proposition} \label{prop:exactfact} 
Suppose that  $G = KN$ is an exact factorization. Then an algebra $A_{M(g,\rho)}$  as in Lemma \ref{lem:AM} is $\kk$-commutative if and only if  the following conditions are simultaneously satisfied:
\begin{enumerate}
\smallskip
\item[\textnormal{(a)}] $\dim \rho = 1$ (hence, $\alpha\psi^{-1}_g$ is a coboundary on $K \cap gLg^{-1}$);
\smallskip
\item[\textnormal{(b)}]  the group $N \cap hLh^{-1}$ is abelian for all $h \in G$, and $\gamma\psi^{-1}_h$ is a coboundary on this group; and 
\smallskip
\item[\textnormal{(c)}]  $|N \cap hLh^{-1}| \cdot |K \cap gLg^{-1}| ~=~ |L|$ for each $h \in G$.
\end{enumerate}
\end{proposition}

\begin{proof}
Recall from Proposition~\ref{prop:inequality} that  $m_{\gamma, \psi}(h)$ denotes the \red{number} of irreducible projective representations of the group $N \cap hLh^{-1}$ with Schur multiplier $\gamma\psi^{-1}_h$. So, using the Artin-Wedderburn theorem and Lemma~\ref{lem:exactfact} we have 
$$\sum_{h \in N \backslash G / L} m_{\gamma, \psi}(h) ~\leq~ \sum_{h \in N \backslash G / L} |N \cap hLh^{-1}| ~\leq~ \frac{|K| \; |L|}{|K \cap gLg^{-1}|^2}  ~\leq~ \frac{|K| \; |L|}{|K \cap gLg^{-1}|^2}(\dim \rho)^2.$$
Moreover, we know by Proposition~\ref{prop:inequality} that $A_{M(g,\rho)}$ is $\kk$-commutative if and only if all three inequalities are equalities. Now the first and third inequalities are equalities if and only if, respectively, conditions~(b) and~(a) hold. By Lemma~\ref{lem:exactfact} the second inequality is an equality if and only if condition~(c) holds.
\end{proof}

Now we finish Example~\ref{ex:H8gt} and the work in Section~\ref{sec:H8} below.

\begin{example}  \label{ex:H8path}
Take $G = D_8$, the exact factorization $KN$ for $K = \mathbb{Z}_2 = \langle z \rangle$ and $N = \mathbb{Z}_2 \times \mathbb{Z}_2 = \langle x,y \rangle$, and $\mathcal{C}(D_8, \omega, \mathbb{Z}_2, 1) \overset{\otimes}{\sim} {\sf Rep}(H_8)$. For each $\mathcal{M}^{K,1}(L,1) \in \textnormal{Irr}({\sf Mod}(\mathcal{C}))$, we choose $M = M(e, \rho^{K \cap L}_{\textnormal{triv}})$ as in the proof of Proposition~\ref{prop:H8match}, and we study the $\kk$-commutativity of the indecomposable semisimple algebra $A_M \in \mathcal{C}$ via Proposition~\ref{prop:exactfact} as follows. Note that by our choice of $M$ we always have that condition (a) of Proposition~\ref{prop:exactfact} holds.

\brk 

For $L = \langle e \rangle$, we have that $K \cap L =  \langle e \rangle$ and that $N \cap h L h^{-1} =  \langle e \rangle$ for all $h \in G$. Therefore, conditions (b) and (c) of Proposition~\ref{prop:exactfact} hold, and $A_M$ is $\kk$-commutative in this case.

\brk 

For $L = \langle x \rangle$ or $\langle xy \rangle$, we have that $K \cap L =  \langle e \rangle$ and that $|N \cap h L h^{-1}| = 2 $ for all $h \in G$. So, Proposition~\ref{prop:exactfact}(b,c) hold, and $A_M$ is $\kk$-commutative in these cases.

\brk 

For $L = \langle z \rangle$, we have that $K \cap L =  \langle z \rangle$ and that $|N \cap h L h^{-1}| = 1 $ for all $h \in G$. So,  $A_M$ is $\kk$-commutative.

\brk 

For $L = \langle x, y \rangle$, we have that $K \cap L =  \langle e \rangle$ and that $N \cap h L h^{-1} = N $ for all $h \in G$. Take $h = e$, and recall from Section~\ref{sec:H8} that $\gamma = \mu$ of \eqref{eq:Z2xZ2}.  
Since $\omega|_N$ is trivial, according to~\eqref{eq:aijg} we get
$$\mu \mu^{-1}_e(x^{i_1}y^{j_1}, x^{i_2}y^{j_2}) 
~=~ \mu(x^{i_1}y^{j_1}, x^{i_2}y^{j_2})
\cdot \mu(x^{-i_2}y^{-j_2}, x^{-i_1}y^{-j_1})
~=~ (-1)^{j_1 i_2 + i_1 j_2},
$$
which is a 2-cocycle on $N$ cohomologous to $\mu$. 
Therefore, $\mu \mu^{-1}_e$ is not a coboundary on $N$. So $A_M$ is not $\kk$-commutative in this case, as Proposition~\ref{prop:exactfact}(b) fails.

\brk 

For $L = \langle xy, z \rangle$, we have that $K \cap L =  \langle z \rangle$ and that $|N \cap h L h^{-1}| = 2 $ for all $h \in G$. So, Proposition~\ref{prop:exactfact}(b,c) hold, and $A_M$ is $\kk$-commutative in this case.

A complete count of the indecomposable bimodules for each pair of algebras above, and thus classification of minimal faithful $\Rep(H_8)$-path algebras, will be carried out in future work.
\end{example}

\section*{Acknowledgments}
The authors would like to thank Dmitri Nikshych and the anonymous referee for pointing out an error in the original version of Theorem~\ref{thm:param}(ii). We would also like to thank Nikshych for a reference for the 3-cocycle in Example~\ref{ex:D8nontriv}. The last author would like to thank Jonathan Beardsley for helpful comments on a preliminary version of this work.
P. Etingof \red{was partially supported} by the US National Science Foundation grant \#DMS-1502244. C. Walton \red{was partially supported} by a research fellowship from the Alfred P. Sloan foundation, and by the US National Science Foundation grants \#DMS-1663775, 1903192.

\bibliography{quiversemisimple}

\end{document}